\documentclass{article}
\usepackage{geometry}
\usepackage{float}
\usepackage{latexsym,bm}
\usepackage{amsmath,amsfonts,amsthm,amssymb,bbding,mathrsfs,mathtools}
\usepackage{graphicx}
\usepackage{hyperref}

\newtheorem{theorem}{Theorem}[section]
\newtheorem{prob}{Problem}[section]

\newtheorem{lemma}{Lemma}[section]

\newtheorem{conj}{Conjecture}[section]
\newtheorem{claim}{Claim}[section]
\newtheorem{definition}{Definition}[section]
\newcommand{\ex}{\mathrm{ex}}

\title{Planar Tur\'an number for balanced double stars}

\author{
Xin Xu\thanks{School of Sciences, North China University of Technology, Beijing, China.}
\and Qiang Zhou\thanks{Academy of Mathematics and Systems Science, Chinese Academy of Sciences, Beijing, China, and University of Chinese Academy of Sciences, Beijing, China.}
\and Tong Li\footnotemark[2]
\and Guiying Yan\footnotemark[2]}

\date{}

\begin{document}
\begin{sloppypar}

\maketitle

\begin{abstract}
Planar Tur\'an number, denoted by $\ex_{\mathcal{P}}(n,H)$, is the maximum number of edges in an $n$-vertex planar graph which does not contain $H$ as a subgraph. Ghosh, Gy\H{o}ri, Paulos and Xiao initiated the topic of the planar Tur\'an number for double stars. For balanced double star, $S_{3,3}$ is the only remaining graph need to be considered. In this paper, we give the exact value of $\ex_{\mathcal{P}}(n,S_{3,3})$, forcing the planar Tur\'an number for all balanced double stars completely determined.
\end{abstract}

\maketitle

\section{Introduction}

All graphs considered in this paper are finite, undirected and simple. Let $V(G)$, $v(G)$, $E(G)$, $e(G)$, $\delta(G)$ and $\Delta(G)$ denote the vertex set, number of vertices, edge set, number of edges, minimum degree and maximum degree of a graph $G$, respectively. For any subset $S\subset V(G)$, the subgraph induced on $S$ is denoted by $G[S]$. We denote by $G\backslash S$ the subgraph induced on $V(G)\backslash S$. If $S=\{v\}$, we simply write $G\backslash v$. We use $e[S,T]$ to denote the number of edges between $S$ and $T$, where $S$, $T$ are subsets of $V(G)$.

Let $H$ be a graph, a graph is called $H$-free if it does not contain $H$ as a subgraph. One of the most classical problems in extremal graph theory, nowadays so-called Tur\'an-type
problem is:

\begin{prob}[Tur\'an Problem]\label{TuranProblem}
    What is the maximum number of edges in an $n$-vertex $H$-free graph $G$?
\end{prob}

We use $\ex(n,H)$ to denote the maximum number of edges in an $n$-vertex $H$-free graph. In 1941, Tur\'an~\cite{turan} gave the exact value of $\ex(n,K_{r})$ and the extremal graph, where $K_{r}$ is a complete graph with $r$ vertices. Later in 1946, Erd\H{o}s and Stone~\cite{erdos1946} extended this result by replacing $K_{r}$ by an arbitrary graph $H$ and showed that $\ex(n,H)=(1-\frac{1}{\chi(H)-1}+o(1))\binom{n}{2}$, where $\chi(H)$ denotes the chromatic number of $H$. This is latter called the ``fundamental theorem of extremal graph theory"~\cite{bollobs2002}.

In 2016, Dowden~\cite{dowden2016} initiated the study of Tur\'an-type problems when host graphs are planar graphs:

\begin{prob}[Planar Tur\'an Problem]\label{PlanarTuranProblem}
    What is the maximum number of edges in an $n$-vertex $H$-free planar graph $G$?
\end{prob}

We use $\ex_{\mathcal{P}}(n,H)$ to denote the maximum number of edges in an $n$-vertex $H$-free planar graph. Dowden studied the planar Tur\'an number of $C_{4}$ and $C_{5}$, where $C_{k}$ is a cycle with $k$ vertices. Ghosh, Gy\H{o}ri, Martin, Paulos and Xiao~\cite{ghosh2022c6} gave the exact value for $C_{6}$. Shi, Walsh and Yu~\cite{shi2023c7}, Gy\H{o}ri, Li and Zhou~\cite{győri2023c7} gave the exact value for $C_{7}$. The planar Tur\'an number of $C_{k}$ is still unknown for $k\geq 8$. Cranston, Lidick\'{y}, Liu and Shantanam~\cite{daniel2022counterexample} first gave both lower and upper bound for general cycles, Lan and Song~\cite{lan2022improved} improved the lower bound. Recently, Shi, Walsh and Yu~\cite{shi2023dense} improved the upper bound, Gy\H{o}ri, Varga and Zhu~\cite{győri2023new} gave a new construction and improved the lower bound. Lan, Shi and Song~\cite{lan2019hfree} gave a sufficient condition for graphs with planar Tur\'an number $3n-6$. We refer the interested readers to more results on paths, theta graphs and other graphs~\cite{lan2019shortpaths,lan2019thetafree,ghosh2023theta6,győri2022extremal,zhai2022,fang2023,fang2022intersecting,li2024,lan2024planar,du2021,győri2023k4c5k4c6}.

\begin{definition}
    A \textbf{double star} $S_{m,l}$ is the graph obtained by taking an edge $xy$ and joining $x$ with $m$ vertices, $y$ with $l$ vertices which are different from the $m$ vertices. 
\end{definition}

\begin{figure}[ht]
    \centering
    \includegraphics[height=4cm, width=7cm]{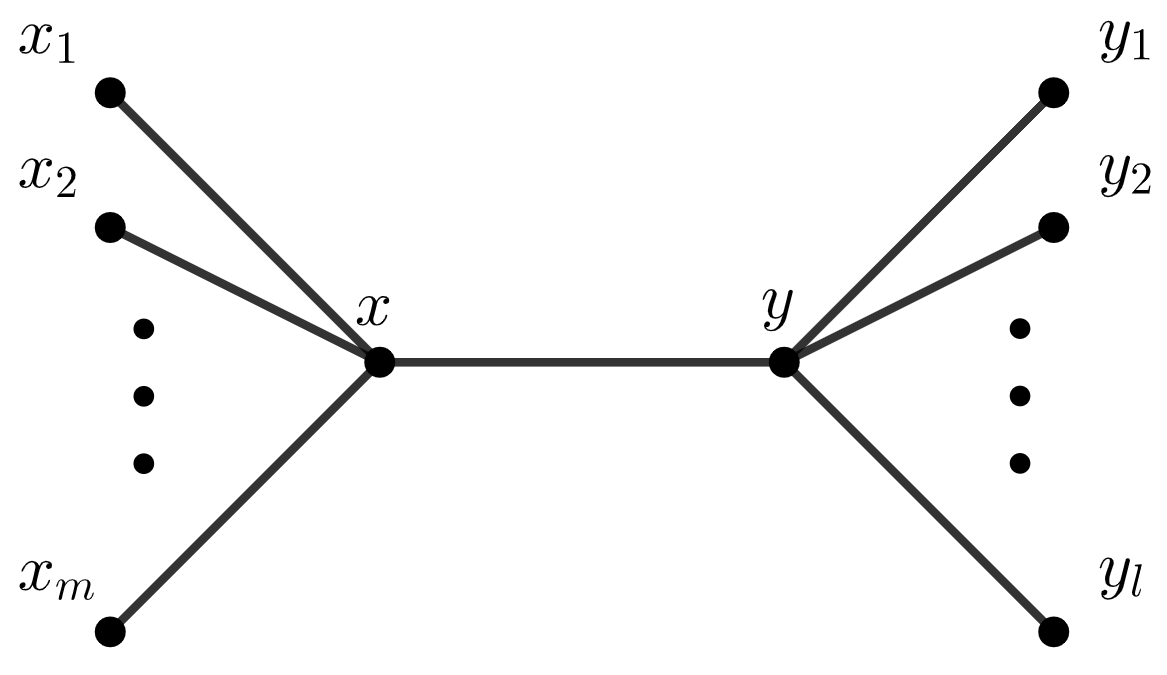}
    \caption{The double star $S_{m,l}$.}
    \label{double_star}
\end{figure}

In 2022, Ghosh, Gy\H{o}ri, Paulos and Xiao~\cite{ghosh2022planar} studied the planar Tur\'an number for $S_{2,2}$, $S_{2,3}$, $S_{2,4}$, $S_{2,5}$, $S_{3,3}$ and $S_{3,4}$. Moreover, they gave the exact value for $S_{2,2}$ and $S_{2,3}$. Later, The first author of this paper  improved the upper bound for $S_{2,5}$~\cite{xu2024improved}.

We say a double star is a \textbf{balanced double star} if $m=l$. For $m\geq 4$, $\ex_{\mathcal{P}}(n,S_{m,m})=3n-6$ since a double wheel graph (a graph with two non-adjacent vertices connecting to all vertices on the cycle $C_{n-2}$) does not contain $S_{m,m}$ as a subgraph. Since $S_{1,1}$ is a path on 4 vertices and every graph without a path on 4 vertices must be a planar graph, the planar Tur\'an number of $S_{1,1}$ is equal to its Tur\'an number. By the result of Faudree and Schelp~\cite{faudree1975}, $\ex_{\mathcal{P}}(n,S_{1,1})=\ex(n,S_{1,1})\leq n$ and the equality holds for $3|n$. Ghosh, Gy\H{o}ri, Paulos and Xiao~\cite{ghosh2022planar} proved that $\ex_{\mathcal{P}}(n,S_{2,2})=2n-4$ when $n\geq 16$  and $\lfloor5n/2\rfloor-5\leq \ex_{\mathcal{P}}(n,S_{3,3})\leq \lfloor5n/2\rfloor-2$ when $n\geq 3$. Moreover, they gave the following conjecture:

\begin{conj}\label{conj}
    \begin{align*}
        \begin{split}
            \ex_{\mathcal{P}}(n,S_{3,3}) = \left\{
                \begin{array}{ll}
                3n-6 & \text{if } 3\leq n \leq 7,\\
                16 & \text{if } n=8,\\
                18 & \text{if } n=9,\\
                \lfloor5n/2\rfloor-5 & \text{otherwise}.
                \end{array}
            \right.
        \end{split}
    \end{align*}
\end{conj}

In this paper, we solve this conjecture by a new method and thus the planar Tur\'an number for all balanced double stars is determined:

\begin{theorem}\label{thm}
    For any integer $n\geq 3$, we have
    \begin{align*}
        \begin{split}
            \ex_{\mathcal{P}}(n,S_{m,m}) = \left\{
                \begin{array}{ll}
                n & \text{if } m=1 \text{ and } 3|n,\\
                n-1 & \text{if } m=1 \text{ and } 3\nmid n,\\
                2n-4 & \text{if } m=2 \text{ and } n\geq 16,\\
                \lfloor5n/2\rfloor-5 & \text{if } m=3 \text{ and }n\geq 10,\\
                3n-6 & \text{if }
                m\geq 4.
                \end{array}
            \right.
        \end{split}
    \end{align*}
\end{theorem}

\section{Definitions and Preliminaries}

Let $G$ be an $S_{3,3}$-free planar graph. For the sake of brevity and simplicity, we always omit the floor function.

\begin{definition}
A $\boldsymbol{k\text{-}l}$ \textbf{edge} is an edge whose end vertices are of degree $k$ and $l$. 
A $\boldsymbol{k\text{-}l\text{-}s}$ \textbf{path} is an induced path consisting of three vertices with degree $k$, $l$ and $s$.
\end{definition}

\begin{definition}
A $\boldsymbol{k\text{-}s}$ \textbf{star} is a subgraph in $G$ on $k+1$ vertices, where there is a central vertex connecting to the other $k$ vertices, and all other $k$ vertices have degree of $s$. 
We use $\boldsymbol{k^{+}\text{-}s^{-}}$ \textbf{star} to denote the one with the central vertex of degree at least $k$ and the other $k$ vertices of degree at most $s$.
\end{definition}


\begin{definition}
Given two subgraphs $H, H'\subset G$, we use $\boldsymbol{H+H'}$ to denote the subgraph induced on $V(H)\cup V(H')$. If $V(H')=\{v\}$, we abbreviate this as $H+v$.
\end{definition}

\begin{definition}
The \textbf{elementary star-block B} of $G$ is one of $5^{+}$-$3^{-}$ star, $5$-$4^{-}$ star, $6$-$6$ edge, $6$-$5$ edge, $6$-$4$ edge, $5$-$5$ edge, $5$-$4$-$5$ path in $G$. Given an elementary star-block B, a vertex $v\in V(G)\backslash V(B)$ is called \textbf{potential} if $N(v)\cap V(B) \neq \emptyset$ when $d(v)\leq 3$   or   $|N(v)\cap V(B)|=2$ when $d(v)=4$. If $B$ is an elementary star-block and $V_{p}$ is a set of potential vertices, then $B$ and $B+V_{p}$ are both called \textbf{star-blocks}. For any star-block $B$, the vertices in $\{v\in V(B)| N(v)\cap V(G\backslash B)\neq \emptyset \text{\ for\  } d(v)\leq 3  \text{\ and\ }|N(v)\cap V(G\backslash B)|=2 \text{\ for\ } d(v)=4\}$ are \textbf{peripheral} vertices.
\end{definition}

Note that if a graph has a $k$-$l$ edge with $k\geq 7$ and $l\geq 4$, then $G$ contains an $S_{3,3}$ as a subgraph. So all vertices of degree at least 5 must lie in some star-block defined above.
Next we ensure that for any vertex $v$ with $d(v)\geq 5$, there is exactly one star-block containing $v$. The star-block is the first one by checking in the order of $7^{+}$-$3^{-}$ star,  $6$-$6$ edge, $6$-$5$ edge, $6$-$4$ edge, $6$-$3^{-}$ star, $5$-$5$ edge, $5$-$4$-$5$ path, $5$-$4^{-}$ star.



\begin{definition}
   Let $H$ be a subgraph of $G$. The \textbf{star-block base} $\mathcal{B}$ of $H$ is the set consisting of  star-blocks satisfying:\\
   $(a)$ $V(H)=\bigcup\limits_{B\in \mathcal{B}}V(B)$;\\
   $(b)$ $\forall B, B'\in \mathcal{B}$, if $V(B)\cap V(B')\neq \emptyset$, then all common vertices are peripheral vertices in both $B$ and $B'$.  Moreover, the common peripheral vertices are called \textbf{shared vertices}.
\end{definition}
It should be noticed that for any shared vertex $v$, we have $2\leq d(v)\leq 4$. In fact, if $d(v)=1$, $v$ belongs to one star-block and can not be shared. If $v$ is a shared vertex with $d(v)\geq 5$, then an $S_{3,3}$ is easily found.


\begin{definition}
Let $G=G_{1}+G_{2}$. If $\mathcal{B}$ is a star-block base of $G_{1}$ and any vertex in $V(G_{2})$ has degree at most $4$ in $G$, then we say $G$ has a \textbf{star-block partition}.
Let $d_{\mathcal{B}}(v)$ be the number of star-blocks in base $\mathcal{B}$ containing $v$.
For any star-block $B\in \mathcal{B}$, let $d_{s}(B)$ denote the number of shared vertices of degree at most $3$ in $B$ and let $d_{s'}(B)$ denote the number of shared vertices of degree $4$ in $B$.
\end{definition}

It is easy to see that $G$ must have a star-block partition. Specially, if $\Delta(G)\leq 4$, then $G=G_{2}$.

\begin{definition}
Let $H$ be a subgraph of $G$. The \textbf{primary weight} of $H$, denoted by $\boldsymbol{w_{0}(H)}$, is defined as
$$\boldsymbol{w_{0}(H)}\coloneqq e(H)+\frac{1}{2}(e[H,G\backslash H])=\frac{1}{2}\sum_{v\in V(H)}d(v).$$
\end{definition}

\begin{definition}
Let $G$ have a star-block partition and a star-block base $\mathcal{B}$. For any $B\in \mathcal{B}$,
the \textbf{modified weight} of $B$, denoted by $\boldsymbol{w(B)}$, is defined as
$$\boldsymbol{w(B)}\coloneqq w_{0}(B)+\frac{s}{2}+\frac{s'}{4}+\mathbf{1}_{B},$$
where $s=d_{s}(B)$, $s'=d_{s'}(B)$, and
\begin{align*}
    \begin{split}
        \mathbf{1}_{B} = \left\{
        \begin{array}{ll}
        1 &\text{if } d_{\mathcal{B}}(v)=3 \text{ for some } v\in B,\\
        0 & \text{otherwise}.\\
        \end{array}
        \right.
    \end{split}
\end{align*}
\end{definition}

Obviously, if there is a star-block partition $G=G_{1}+G_{2}$, then $e(G)=w_{0}(G)=w_{0}(G_{1})+w_{0}(G_{2})$.

\begin{definition}
    Let $G$ have a star-block base $\mathcal{B}=\{B_{i}, i=1,2,\cdots, T\}$ and $\mathcal{B}'$ be the star-block base obtained from $\mathcal{B}$ satisfying: $(a)$ $|\mathcal{B}'|=|\mathcal{B}|$; $(b)$ for all $1\leq i\leq T$, there exist $B_{i}\in \mathcal{B}$ and $B'_{i}\in \mathcal{B}'$ such that $V(B_{i})\subset V(B'_{i})$ and $w(B'_{i})/v(B'_{i})\leq w(B_{i})/v(B_{i})$, where at least one inequality is strict.
    Then we say $\mathcal{B}'$ is the \textbf{refinement} of $\mathcal{B}$.
\end{definition}

For the sake of convenience in subsequent discussion, we categorize the star-blocks into three types:

\begin{itemize}
\item[$\bullet$] $\mathcal{B}_{0}\coloneqq \{B\in \mathcal{B}\ |\ d_{s}(B)+d_{s'}(B)=0\}$,

\item[$\bullet$] $\mathcal{B}_{1}\coloneqq \{B\in \mathcal{B}\ |\ d_{s}(B)+d_{s'}(B)\geq 1 \ and\  d_{\mathcal{B}}(v)\leq 2 \ for\ each\ v\in B\}$,

\item[$\bullet$] $\mathcal{B}_{2}\coloneqq \{B\in \mathcal{B}\ |\ d_{s}(B)+d_{s'}(B)\geq 1\ and\ d_{\mathcal{B}}(v)=3\ for\ some\ v\in B\}$.
\end{itemize}

We show that $w(B)$ can be constrained for each star-block in some base $\mathcal{B}$. 

\begin{lemma}\label{lemma_w}
Let $G$ be an $S_{3,3}$-free planar graph on $n$ vertices. Then there exists a star-block partition $G=G_{1}+G_{2}$ such that all star-blocks in the base $\mathcal{B}$ satisfying:
    \begin{align*}
        \begin{split}
            w(B) \leq \left\{
            \begin{array}{ll}
            \frac{5}{2}v(B)-\frac{5}{2} & \text{when }B\in \mathcal{B}_{0},\\
            \frac{5}{2}v(B)-\frac{5}{t} & \text{when }B\in \mathcal{B}_{1},\\
            \frac{5}{2}v(B)-1 & \text{when }B\in \mathcal{B}_{2},
            \end{array}
            \right.
        \end{split}
    \end{align*}
where  $t=|\mathcal{B}_{1}|+|\mathcal{B}_{2}|$.
\end{lemma}

This lemma is the highlight of this paper, which make it possible to prove the following two lemmas.

\begin{lemma}\label{lemma_eq1}
    Let $G$ be an $S_{3,3}$-free planar graph on $n\geq 7$ vertices. If $G$ contains only one star-block, then
    \begin{align*}
        \begin{split}
            e(G) \leq \left\{
            \begin{array}{ll}
                15 & \text{when }n=7,\\
                16 & \text{when }n=8,\\
                18 & \text{when }n=9,\\
                5n/2-5 & \text{when }n\geq 10.
            \end{array}
            \right.
        \end{split}
    \end{align*}

\end{lemma}

\begin{lemma}\label{lemma_eq2}
Let $G$ be an $S_{3,3}$-free planar graph with a star-block partition $G=G_{1}+G_{2}$ and a star-block base $\mathcal{B}$. If $|\mathcal{B}|\geq 2$, then $e(G)\leq 5n/2-5$. 
\end{lemma}

Together with the extremal graphs achieving the bound, Theorem~\ref{thm} can be deduced from these two lemmas, which will be proved in the following sections. 

\section{Proof of Lemma~\ref{lemma_w}}

Suppose that there exists a partition  
$G=G_{1}+G_{2}$ and a star-block base $\mathcal{B}$. We show that if there exists a star-block $B\in \mathcal{B}$ which does not satisfy the corresponding upper bound, we can construct a refinement $\mathcal{B}'$ such that for some $B'\in \mathcal{B}'$, $B\subset B'$ and $w(B')$ satisfies the bound. 

There are several different star-blocks in $G$, such as $5^{+}$-$3^{-}$ star, $6$-$6$ edge, $6$-$5$ edge, $6$-$4$ edge, $5$-$5$ edge, $5$-$4$-$5$ path, $5$-$4^{-}$ star and their variant forms.

Now we consider each case in turn. Recall that $s$ the number of shared vertices of degree at most $3$ and $s'$ denote the number of shared vertices of degree $4$ in $B$, $\mathbf{1}_{B}$ the characteristic function for whether $B$ contains a $3$-degree vertex with $d_{\mathcal{B}}(v)=3$ or not.

\noindent\textbf{Case 1.} $B$ is a $5^{+}$-$3^{-}$ star.

Assume that $B$ is a $k$-$3^{-}$ star for $k\geq 5$. We have 
\begin{align*}
    w(B)&\leq \frac{1}{2}(k+3k)+\frac{s}{2}+\mathbf{1}_{B}\\
    &\leq \frac{5}{2}k+\mathbf{1}_{B}.
\end{align*}

If $\mathbf{1}_{B}=0$, then $w(B)\leq \frac{5}{2}k=\frac{5}{2}(k+1)-\frac{5}{2}$. If $\mathbf{1}_{B}=1$, then $w(B)\leq \frac{5}{2}k+1=\frac{5}{2}(k+1)-\frac{3}{2}<\frac{5}{2}(k+1)-1$. 

\noindent\textbf{Case 2.} $B$ is a $6$-$6$ edge. 

Let $uv$ be the $6$-$6$ edge. There exists at least $5$ triangles sitting on the edge $uv$, otherwise an $S_{3,3}$ is found in $G$. Let $a_{1}, a_{2}, a_{3}$, $a_{4}$ and $a_{5}$ be the vertices adjacent to both $u$ and $v$, as shown in Figure~\ref{fig_66}$(a)$. Let $S_{1}=\{a_{1}, a_{2}, a_{3}, a_{4}, a_{5}\}$ and $H_{1}=G[S_{1}]$.

\begin{figure}[ht]
  \centering
  \includegraphics[width=0.55\textwidth]{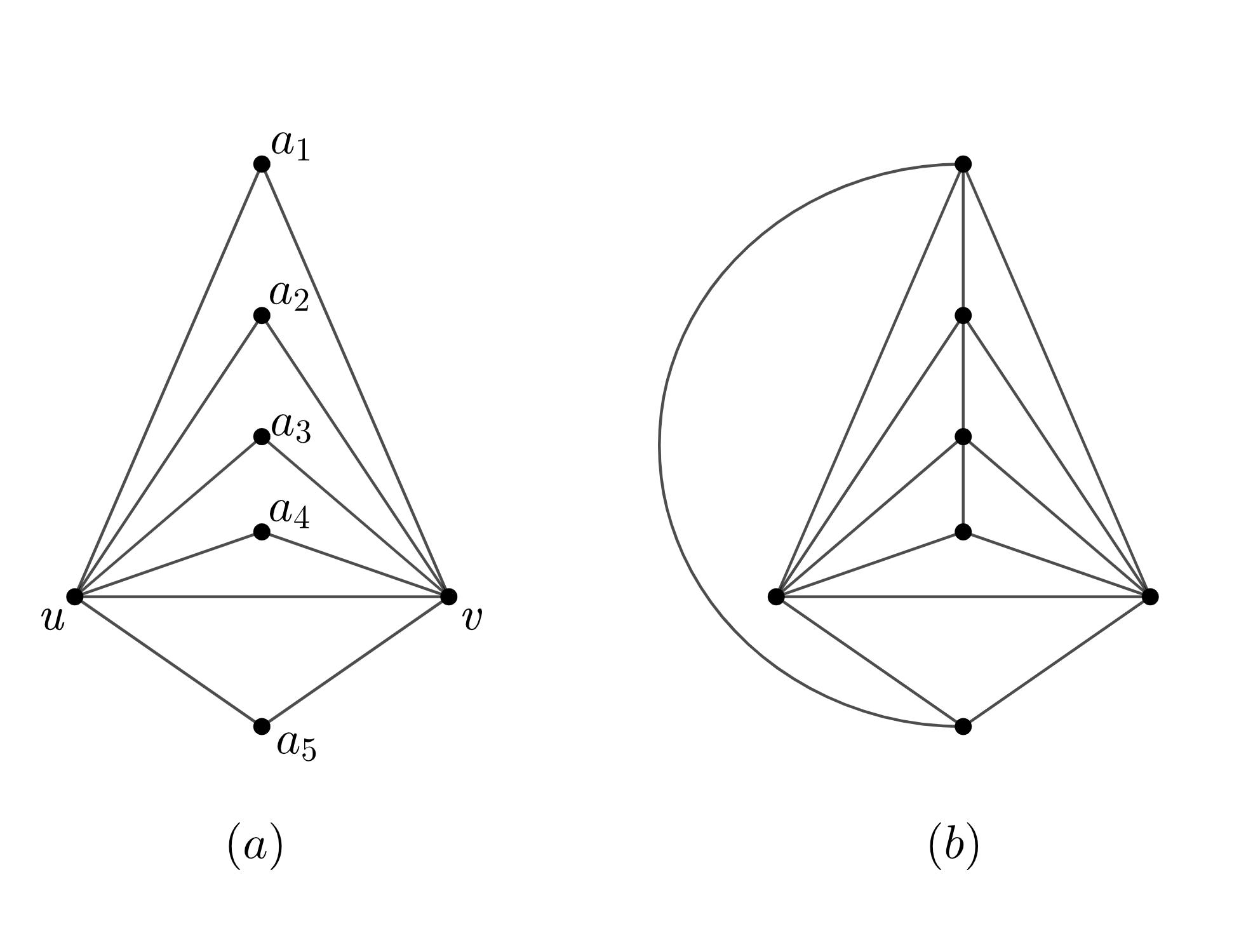}
  \caption{$(a)$ The star-block is a $6$-$6$ edge $uv$. $(b)$ The connected component on $7$ vertices.}
  \label{fig_66}
\end{figure}

Note that vertices in $S_{1}$ can form a path of length at most $4$ and each vertex in  $S_{1}$ can have at most one neighbor in $V(G)\backslash B$, otherwise $G$ contains an $S_{3,3}$. We have $e(H_{1})\leq 4$ and $e[H_{1},G\backslash B]\leq 5$. Moreover, for any $x\in S_{1}$, if $d(x)=4$, then $x$ can not have a neighbor outside of $B$, otherwise an $S_{3,3}$ is also found. So the degree of shared vertices in $S_{1}$ is exactly $3$, which implies $s'=0$. Each shared vertex is shared by exactly two star-blocks since it is adjacent to both $u$ and $v$, which means $\mathbf{1}_{B}=0$.

Hence we have $w(B)=11+e(H_{1})+\frac{1}{2}e[H_{1},G\backslash B]+\frac{s}{2}$.

Assume that $e(H_{1})=0$. Now we analyze the relationship between the number of shared vertices in $B$ and the number of star-blocks in $G$. It can be checked that
\begin{align*}
    \begin{split}
            t \geq \left\{
            \begin{array}{ll}
                0 & \text{when }s=0,\\
                2 & \text{when }s=1, 2,\\
                3 & \text{when }s=3, 4,\\
                4 & \text{when }s=5.
            \end{array}
            \right.
        \end{split}
\end{align*}

When $s=0$, $w(B)\leq 11+\frac{5}{2}= \frac{5}{2}\cdot 7-4<\frac{5}{2}\cdot 7-\frac{5}{2}$.

When $s=1, 2$, $w(B)\leq 11+\frac{5}{2}+\frac{s}{2}\leq \frac{5}{2}\cdot 7-3<\frac{5}{2}\cdot 7-\frac{5}{t}$. 

When $s=3, 4$, $w(B)\leq 11+\frac{5}{2}+\frac{s}{2}\leq \frac{5}{2}\cdot 7-2<\frac{5}{2}\cdot 7-\frac{5}{t}$.

When $s=5$, $w(B)\leq 11+\frac{5}{2}+\frac{s}{2}\leq \frac{5}{2}\cdot 7-\frac{3}{2}<\frac{5}{2}\cdot 7-\frac{5}{t}$.

If $e(H_{1})=1$, assume that $a_{1}a_{2}\in E(G)$. Then $a_{1}, a_{2}$ can not have a neighbor outside and can not be shared either. This implies $s\leq 3$ and $e[H_{1},G\backslash B]\leq 3$. Then if $s=0$, we have $w(B)\leq 12+\frac{3}{2}= \frac{5}{2}\cdot 7-4<\frac{5}{2}\cdot 7-\frac{5}{2}$. If $s\geq 1$, then $w(B)\leq 12+\frac{3}{2}+\frac{s}{2}\leq \frac{5}{2}\cdot 7-\frac{5}{2}\leq \frac{5}{2}\cdot 7-\frac{5}{t}$.

Similarly, if $e(H_{1})=p$ for $2\leq p\leq 4$, we obtain that $s\leq 4-p$ and $e[H_{1},G\backslash B]\leq 4-p$. It follows that $w(B)\leq 11+p+\frac{4-p}{2}+\frac{4-p}{2}=\frac{5}{2}\cdot 7-\frac{5}{2}$. 

Specially, if $s=0$, $w(B)\leq 13+\frac{p}{2}$. When $p\leq 3$, $w(B)\leq \frac{5}{2}\cdot 7-3$. When $p=4$, $B$ is a connected component in $G$ with $e(B)=w(B)=15$, as shown in Figure ~\ref{fig_66}$(b)$. Here, we list these two results separately because they will play a role in the proof of the Lemma~\ref{lemma_eq1}.

\noindent\textbf{Case 3.} $B$ is a $6$-$5$ edge.

Let $uv$ be the $6$-$5$ edge and $u$ be the vertex of degree 6. There exist at least $4$ triangles sitting on the edge $uv$, otherwise an $S_{3,3}$ is found. Let $a_{1}, a_{2}, a_{3}$ and $a_{4}$ be the vertices adjacent to both $u$ and $v$, let $b_{1}$ be the vertex adjacent to only $u$, as shown in Figure~\ref{fig_65}$(a)$. Let $S_{1}=\{a_{1}, a_{2}, a_{3}, a_{4}\}$, $S_{2}=\{b_{1}\}$ and $H_{1}=G[S_{1}]$.

\begin{figure}[ht]
  \centering
  \includegraphics[width=0.8\textwidth]{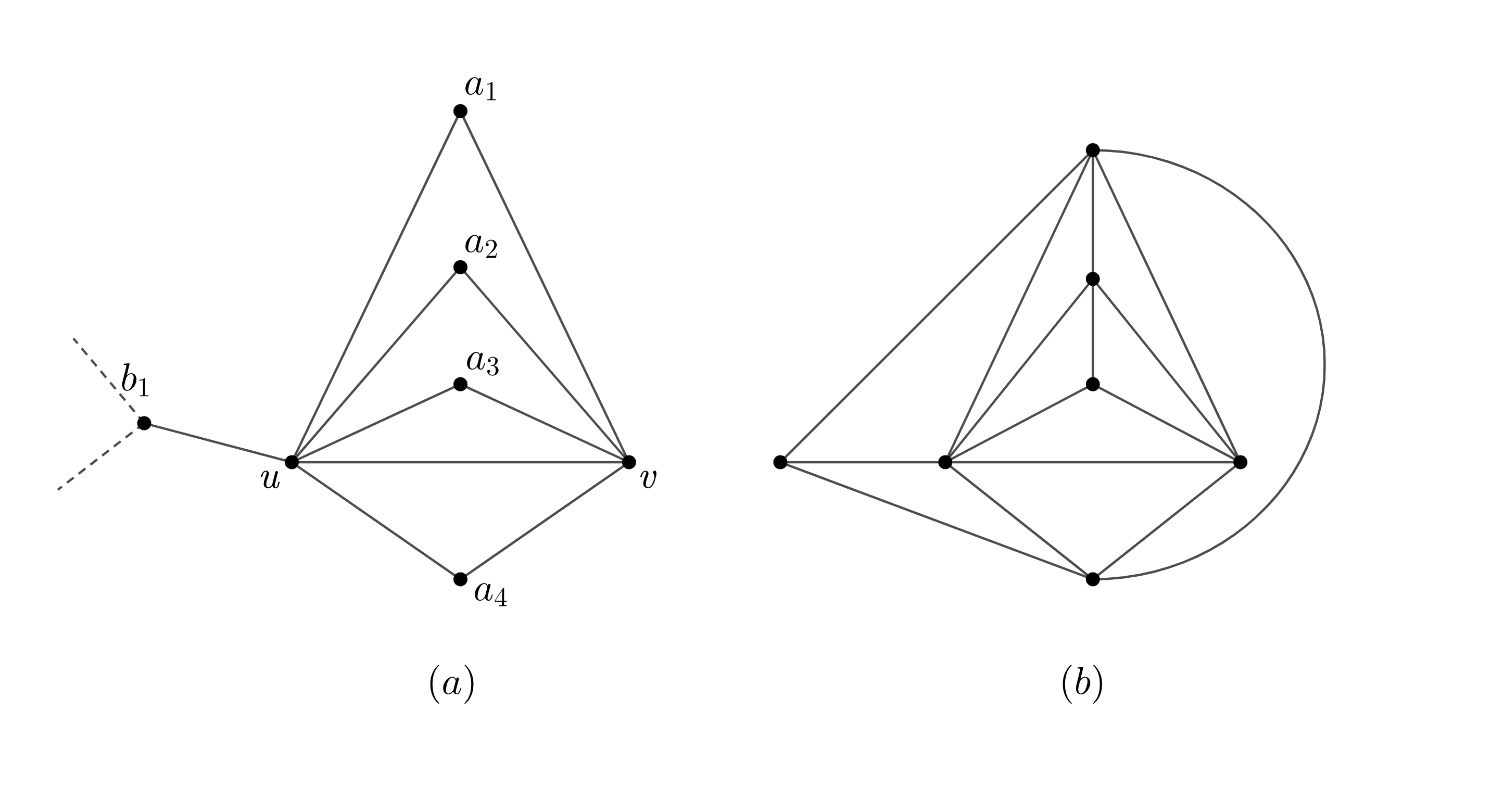}
  \caption{$(a)$ The star-block is a $6$-$5$ edge $uv$. $(b)$ The connected component on $7$ vertices.}
  \label{fig_65}
\end{figure}

Similar to Case 2, each shared vertex in $S_{1}$ has degree exactly 3. We claim that $d(b_{1})\leq 3$. In fact, if $d(b_{1})\geq 4$, there exists an edge $b_{1}b'_{1}\in E(G)$ and $b'_{1}\notin S_{1}$, then $G$ contains an $S_{3,3}$.

\noindent\textbf{Case 3.1.}  If $d_{\mathcal{B}}(b_{1})=3$, we have $\mathbf{1}_{B}=1$ and $b_{1}$ is not adjacent to any vertex in $S_{1}$. Then $w(B)= 10+e(H_{1})+\frac{1}{2}e[B,G\backslash B]+\frac{s}{2}+1$. Assume that $e(H_{1})=0$. It follows that $w(B)\leq 10+\frac{6}{2}+\frac{5}{2}+1=\frac{33}{2}=\frac{5}{2}\cdot 7-1$. If $e(H_{1})\geq 1$, it is easy to verify that $w(B)$ will not increase, which satisfies the upper bound too.

\noindent\textbf{Case 3.2.} If $d_{\mathcal{B}}(b_{1})=2$, it is obtained that $\mathbf{1}_{B}=0$.
Here we give the relationship between the number of shared vertices in $B$ and the number of star-blocks in $G$.
When $s=1, 2, 3$, $t\geq 2$. If $e(H_{1})=0$, we have $w(B)\leq 10+\frac{6}{2}+\frac{3}{2}=\frac{29}{2}=\frac{5}{2}\cdot 7-3$. If $e(H_{1})\geq 1$, it is easy to check that $w(B)$ will not increase.
When $s=4, 5$, it is easy to see that $t\geq 3$ and $e(H_{1})=0$. We have $w(B)\leq 10+\frac{6}{2}+\frac{5}{2}=\frac{31}{2}<\frac{5}{2}\cdot 7-\frac{5}{t}$.

\noindent\textbf{Case 3.3.} If $d_{\mathcal{B}}(b_{1})=1$, then $s\leq 4$.

When $s=0$, we discuss the subcases based on the degree of $b_{1}$ in $B$. If $|N(b_{1})\cap S_{1}|=0$, the discussion here is essentially the same as the subcase when $d_{\mathcal{B}}(b_{1})=3$, where $w(B)$ here is reduced by exactly $\mathbf{1}_{B}+\frac{1}{2}$. So $w(B)\leq \frac{5}{2}\cdot 7-\frac{5}{2}$. If $|N(b_{1})\cap S_{1}|=1$, the proof is same as the  subcase when $d_{\mathcal{B}}(b_{1})=2$. And $w(B)$ is reduced by at least $\frac{1}{2}$. Then $w(B)\leq \frac{31}{2}-\frac{1}{2}=\frac{5}{2}\cdot 7-\frac{5}{2}$. In fact,  a more refined calculation can yield that $w(B)\leq  \frac{5}{2}\cdot 7-3$.

It remains to prove the subcase when $|N(b_{1})\cap S_{1}|=2$. Assume that $b_{1}a_{1}, b_{1}a_{4}\in E(G)$. Then $b_{1}, a_{1}, a_{4}$ all have no neighbor outside of $B$, which implies $s\leq 2$. If  $e(H_{1})=p$ for $1\leq p\leq 3$, we obtain that $s\leq 3-p$ and $e[B,G\backslash B]=e[H_{1},G\backslash B]\leq 3-p$. Hence $w(B)\leq 12+p+\frac{3-p}{2}+\frac{3-p}{2}=\frac{5}{2}\cdot 7-\frac{5}{2}$. Specially, it can be checked that the equality holds if and only if $B$ is a connected component, as shown in Figure~\ref{fig_65}$(b)$.

When $s=1, 2$, we obtain $e(H_{1})\leq 2$ and $t\geq 2$. If $e(H_{1})=0$, it follows that $w(B)\leq 10+\frac{6}{2}+\frac{2}{2}=\frac{5}{2}\cdot 7-\frac{7}{2}$. And if $e(H_{1})\geq 1$, it is easy to see that $w(B)$ does not increase.

When $s=3, 4$, $e(H_{1})=0$ and $t\geq 3$. We have that 
$w(B)\leq 10+\frac{6}{2}+\frac{4}{2}=\frac{5}{2}\cdot 7-\frac{5}{2}$.

\noindent\textbf{Case 4.} $B$ is a $6$-$4$ edge.

Let $uv$ be the $6$-$4$ edge, as shown in Figure~\ref{fig_64}. There exist at least $3$ triangles sitting on the edge $uv$, otherwise an $S_{3,3}$ is found. Let $S_{1}=\{a_{1}, a_{2}, a_{3}\}$, $S_{2}=\{b_{1}, b_{2}\}$ and $H_{1}=G[S_{1}]$.

\begin{figure}[ht]
  \centering
  \includegraphics[width=0.55\textwidth]{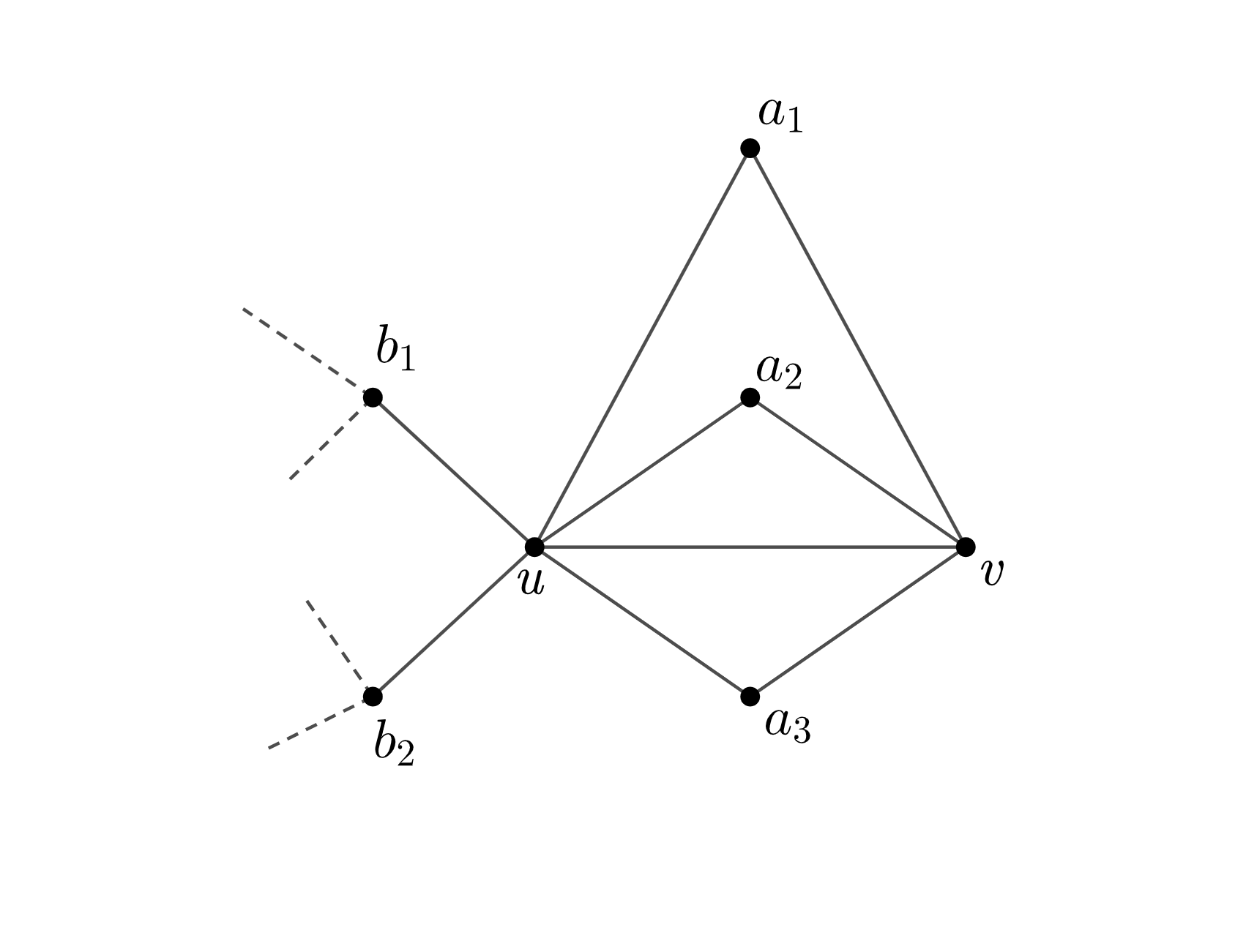}
  \caption{The star-block $B$ is a $6$-$4$ edge $uv$. }
  \label{fig_64}
\end{figure}

\noindent\textbf{Case 4.1.} $d_{\mathcal{B}}(b_{1})=3, d_{\mathcal{B}}(b_{2})\leq 3$. 

Note that $b_{1}$ is not adjacent to any vertex in $S_{1}\cup \{b_{2}\}$. Each vertex in $\{b_{1}, b_{2}\}$ has at most two neighbors outside.

If $e(H_{1})=0$ and $e[b_{2}, S_{1}]=0$, then $w(B)\leq 9+\frac{7}{2}+\frac{5}{2}+1=16< \frac{5}{2}\cdot 7-1$. It is easy to check that the conditions $e(H_{1})\geq 1$ and $e[b_{2}, S_{1}]\geq 1$ both will not lead to an increase in weight $w(B)$.

\noindent\textbf{Case 4.2.} 
$d_{\mathcal{B}}(b_{1})=d_{\mathcal{B}}(b_{2})=2$.

It can be confirmed that any shared vertex in $B$ has degree at most $3$. Thus $d(b_{1}), d(b_{2})\leq 3$. We may assume that there does not exist a $6$-$5$ edge in $B$. Then $d(a_{1}), d(a_{2}), d(a_{3})\leq 4$. The number of shared vertices in $S_{1}$ is $s-2$. It follows that $w(B)=e(H_{1})+\frac{1}{2}e[H_{1},G\backslash H_{1}]+\frac{s}{2}\leq \frac{1}{2}[6+4+4\cdot (5-s)+3\cdot s]+\frac{s}{2}=\frac{30}{2}=\frac{5}{2}\cdot 7-\frac{5}{2}$.

\noindent\textbf{Case 4.3.} $d_{\mathcal{B}}(b_{1})=2, d_{\mathcal{B}}(b_{2})=1$.

Similarly, we have $d(b_{1})\leq 3$ and the number of shared vertices in $S_{1}$ is $s-1$. So 
$w(B)=e(H_{1})+\frac{1}{2}e[H_{1},G\backslash H_{1}]+\frac{s}{2}\leq \frac{1}{2}[6+4\cdot 2+4\cdot (4-s)+3\cdot s]+\frac{s}{2}=\frac{30}{2}=\frac{5}{2}\cdot 7-\frac{5}{2}$.

\noindent\textbf{Case 4.4.} $d_{\mathcal{B}}(b_{1})=d_{\mathcal{B}}(b_{2})=1$.

Similarly, we have $w(B)\leq \frac{1}{2}[6+4\cdot 3+4\cdot (3-s)+3\cdot s]+\frac{s}{2}=\frac{5}{2}\cdot 7-\frac{5}{2}$.

Note that if $s=0$, the equality holds when $B$ is a connected component. But there must exist a vertex  of degree at least $5$ in $B$, which means $B$ is a block star of $6$-$6$ edge or $6$-$5$ edge.

\noindent\textbf{Case 5.} $B$ is a $5$-$5$ edge.

This case is crucial for the proof and is also the most complex part. Let $uv$ be the $5$-$5$ edge. There exist at least $3$ triangles sitting on the edge $uv$, otherwise an $S_{3,3}$ is found. We distinguish the cases based on the number of triangles sitting on $uv$.

\noindent\textbf{Case 5.1.} There are $4$ triangles sitting on $uv$.

Let $a_{1}, a_{2}, a_{3}$ and $a_{4}$ be the vertices adjacent to both $u$ and $v$, as shown in Figure \ref{fig_55}$(a)$. Let $S=\{u, v, a_{1}, a_{2}, a_{3}, a_{4}\}$, $S_{1}=\{a_{1}, a_{2}, a_{3}, a_{4}\}$ and $H_{1}=G[S_{1}]$.

\begin{figure}[ht]
  \centering
  \includegraphics[width=0.9\textwidth]{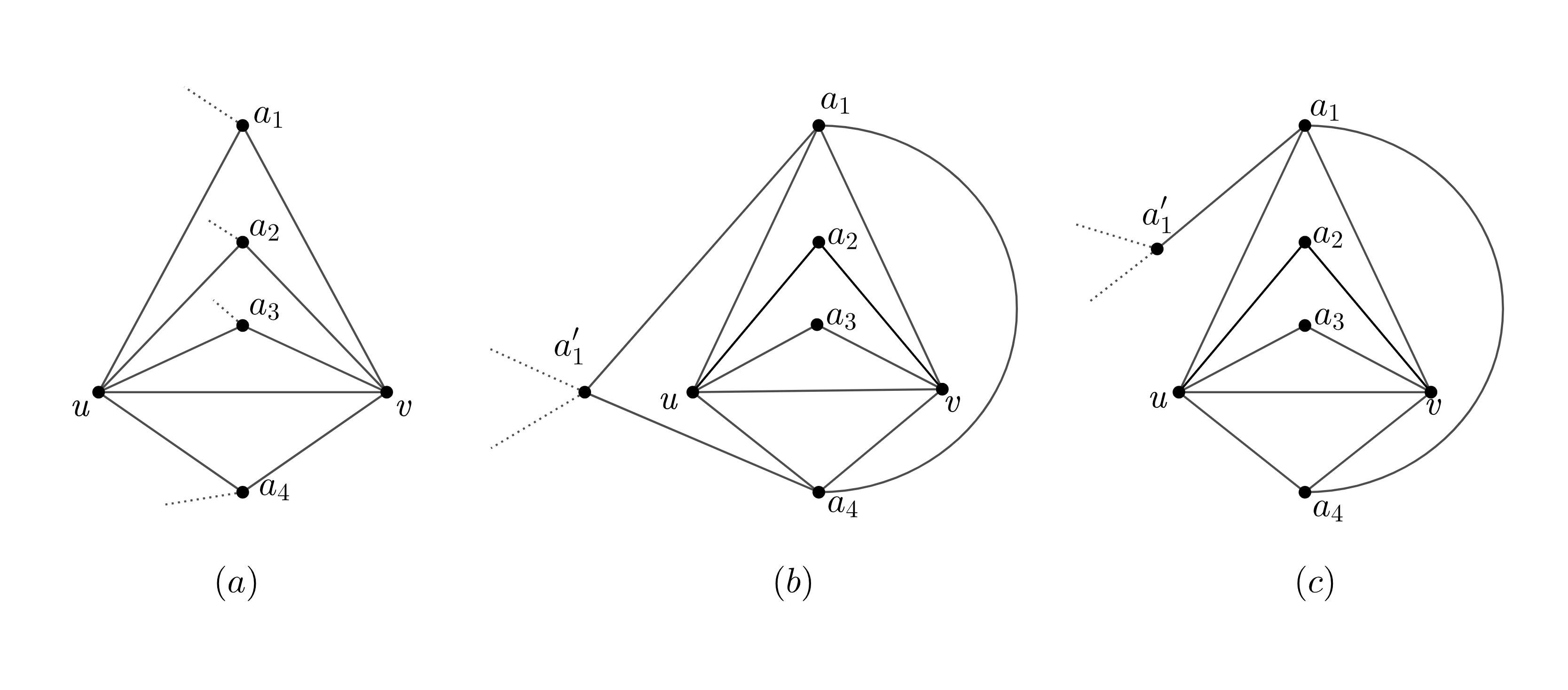}
  \caption{The star-block $B$ is a $5$-$5$ edge on which four triangles sitting.}
  \label{fig_55}
\end{figure}

Note that each vertex in $S_{1}$ can have at most one neighbor outside of $B$, thus $s'=0$ and $\mathbf{1}_{B}=0$.

Assume that $e(H_{1})=0$, then $w(B)\leq 9+\frac{4}{2}+\frac{s}{2}$. If $s=0$, we have $w(B)\leq 11< \frac{5}{2}\cdot 6 -\frac{5}{2}$. If $s=1, 2$, then $t\geq 2$. It follows $w(B)\leq 12< \frac{5}{2}\cdot 6 -\frac{5}{t}$. If $s=3, 4$, then $t\geq 3$, which implies $w(B)\leq 13< \frac{5}{2}\cdot 6 -\frac{5}{t}$.

\noindent\textbf{Case 5.1.1.}  $e(H_{1})=1$.

Without loss of generality, we may assume that $a_{1}a_{4}\in E(G)$. Then $a_{1}, a_{4}$ can not be shared vertices and $s\leq 2$. If there exists one vertex of $S_{1}$ has no neighbor outside, then $w(B)\leq 10+\frac{3}{2}+\frac{s}{2}\leq \frac{5}{2}\cdot 6-\frac{5}{2}.$ Thus each vertex of $S_{1}$ has exactly one neighbor outside. Let $a_{1}a'_{1}\in E(G)$ and $B'=B+a'_{1}$. Then  $\mathcal{B}'$ is the corresponding refinement. It is easy to see that $d(a'_{1})\leq 4$, otherwise there exists an $S_{3,3}$.

When $d(a'_{1})= 4$, we have $a'_{1}a_{4}\in E(G)$, as shown in Figure~\ref{fig_55}$(b)$, otherwise an $S_{3,3}$ is found. If $d_{\mathcal{B}'}(a'_{1})=2$, we have $t\geq 2$ when $s=0$ and $t\geq 3$ when $s=1, 2$. ($s$ is the number of shared vertices in $B$). It follows that $w(B')\leq 12+\frac{4}{2}+\frac{s}{2}+\frac{1}{4}< \frac{5}{2}\cdot 7-\frac{5}{t}$. If $d_{\mathcal{B}'}(a'_{1})=1$, then $w(B')\leq 12+\frac{4}{2}+\frac{s}{2}\leq \frac{5}{2}\cdot 7-\frac{5}{2}$.

When $d(a'_{1})\leq 3$, the star-block is shown  in Figure~\ref{fig_55}$(c)$. Then $w(B')\leq 11+\frac{5}{2}+\frac{s+1}{2}+\mathbf{1}_{B'}$. If $d_{\mathcal{B}'}(a'_{1})=3$, then $w(B')\leq 11+\frac{5}{2}+\frac{s+1}{2}+1< \frac{5}{2}\cdot 7-1$. If $d_{\mathcal{B}'}(a'_{1})\leq 2$ then $w(B')\leq 11+\frac{5}{2}+\frac{s+1}{2}\leq \frac{5}{2}\cdot 7-\frac{5}{2}$. Specially, if there is no shared vertex in $G$, we have $w(B')\leq \frac{27}{2}= \frac{5}{2}\cdot 7 -4$.

\noindent\textbf{Case 5.1.2.} $e(H_{1})=2$.

There are two distinct non-isomorphic subgraphs, as shown in Figure~\ref{fig_55_2}$(a, b)$.

\begin{figure}[ht]
  \centering
  \includegraphics[width=0.9\textwidth]{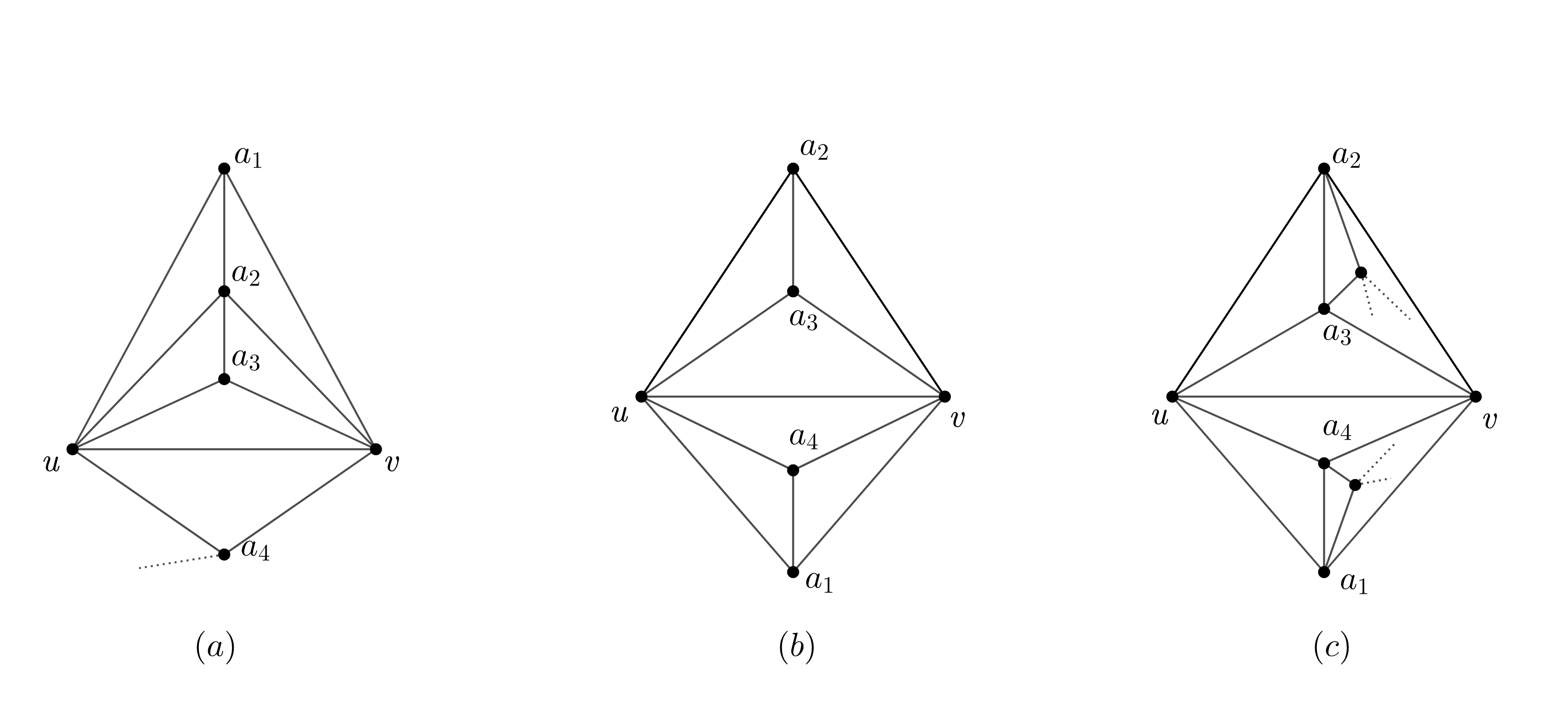}
  \caption{The star-block $B$ is a $5$-$5$ edge with $e(H_{1})=2$. }
  \label{fig_55_2}
\end{figure}

\noindent\textbf{(I).} 
$a_{1}a_{2}, a_{2}a_{3}\in E(G)$. 
It is easy to get $s\leq 1$.

If $d(a_{2})=5$, then there exists an edge $a_{2}a'_{2}\in E(G)$ such that $d(a'_{2})\leq 3$. Let  $B'=B+a'_{2}$ and  $\mathcal{B}'$ be the corresponding refinement. If $d_{\mathcal{B}'}(a'_{2})=3$, then $w(B')\leq 12+\frac{5}{2}+\frac{s+1}{2}+1\leq \frac{5}{2}\cdot 7-1$. If $d_{\mathcal{B}'}(a'_{2})=2$, then $t\geq 2$ when $s=0$ and $t\geq 3$ when $s=1$. It follows $w(B')\leq 12+\frac{5}{2}+\frac{s+1}{2}\leq \frac{5}{2}\cdot 7-\frac{5}{t}$. If $d_{\mathcal{B}'}(a'_{2})=1$, we have $w(B')\leq 12+\frac{5}{2}+\frac{s}{2}\leq \frac{5}{2}\cdot 7-\frac{5}{2}$. Specially, if $s=0$, we have $w(B')\leq 12+\frac{5}{2}= \frac{5}{2}\cdot 7 -3$.

Next we may assume that $d(a_{2})\leq 4$. If there exists one vertex in $\{a_{1}, a_{3}, a_{4}\}$ having no neighbor outside, then $w(B)\leq 11+\frac{2}{2}+\frac{s}{2}\leq \frac{5}{2}\cdot 6-\frac{5}{2}$. Specially, if $s=0$, we have $w(B)\leq 11+\frac{2}{2}= \frac{5}{2}\cdot 6 -3$. So we may assume that each vertex of $\{ a_{1}, a_{3}, a_{4}\}$ has a neighbor in $G\backslash B$. Let $a_{3}a'_{3}\in E(G)$ and $B'=B+a'_{3}$ with $d(a'_{3})\leq 3$. We have $w(B')$ satisfies the upper bound. The proof is similar to $a'_{2}$ mentioned above, so we do not elaborate further here.

\noindent\textbf{(II).} 
$a_{1}a_{4}, a_{2}a_{3}\in E(G)$. It is confirmed that $s=0$.

If each vertex in $S_{1}$ has no neighbor outside of $B$, then $w(B)=11=\frac{5}{2}\cdot 6-4$. Let $a_{2}a'_{2}\in E(G)$ and $B'=B+a'_{2}$. Note that $d(a'_{2})\leq 4$.

When $d(a'_{2})=4$, we have $|N(a'_{2})\cap S_{1}|=2$. Then $a'_{2}$ must be adjacent to $a_{3}$, otherwise an $S_{3,3}$ is found. If $a_{4}$ has no neighbor outside, then $w(B')\leq 13+\frac{3}{2}+\frac{1}{4}< \frac{5}{2}\cdot 7-\frac{5}{2}$.
Let $a_{4}a'_{4}\in E(G)$,  $B^{*}=B'+a'_{4}$ and $\mathcal{B}^{*}$ is the corresponding refinement of $\mathcal{B}'$. Similarly, when $d(a'_{4})=4$, we have $|N(a'_{2})\cap S_{1}|=2$, as shown in Figure~\ref{fig_55_2}$(c)$. 
If $d_{\mathcal{B}^{*}}(a'_{2})=d_{\mathcal{B}^{*}}(a'_{4})=2$, then $t\geq 3$. It follows $w(B^{*})\leq 15+\frac{4}{2}+\frac{2}{4}<  \frac{5}{2}\cdot 8-\frac{5}{t}$. 
Otherwise, $w(B^{*})\leq 15+\frac{4}{2}+\frac{1}{4}<  \frac{5}{2}\cdot 8-\frac{5}{2}$.
When $d(a'_{4})=3$, it can be proved similarly.
If $d_{\mathcal{B}^{*}}(a'_{4})=3$, we get $w(B^{*})\leq 14+\frac{5}{2}+\frac{1}{2}+\frac{1}{4}+1< \frac{5}{2}\cdot 8-1$.
If $d_{\mathcal{B}^{*}}(a'_{4})=2$, then $w(B^{*})\leq 14+\frac{5}{2}+\frac{1}{2}+\frac{1}{4}< \frac{5}{2}\cdot 8-\frac{5}{2}$. 
If $d_{\mathcal{B}^{*}}(a'_{4})=1$, we have $w(B^{*})\leq 14+\frac{5}{2}+\frac{1}{4}< \frac{5}{2}\cdot 8-3$. Specially, if there is no shared vertex in $G$, we have $w(B^{*})\leq 14+\frac{5}{2}= \frac{5}{2}\cdot 8-\frac{7}{2}$.

When $d(a'_{2})\leq 3$, we can also analyze it based on $d_{\mathcal{B}'}(a'_{2})$.
If $d_{\mathcal{B}'}(a'_{2})=3$, we have $w(B')\leq 12+\frac{5}{2}+\frac{1}{2}+1< \frac{5}{2}\cdot 7-1$.
If $d_{\mathcal{B}'}(a'_{2})=2$, then $w(B')\leq 12+\frac{5}{2}+\frac{1}{2}= \frac{5}{2}\cdot 7-\frac{5}{2}$.
If $d_{\mathcal{B}'}(a'_{2})=1$, it follows $w(B')\leq 12+\frac{5}{2}=\frac{5}{2}\cdot 7-3$.

\noindent\textbf{Case 5.1.3.} $e(H_{1})=3$.

Let $a_{1}a_{2}, a_{2}a_{3}, a_{1}a_{4}\in E(G)$, as shown in Figure~\ref{fig_55_3}.

\begin{figure}[ht]
  \centering  \includegraphics[width=0.55\textwidth]{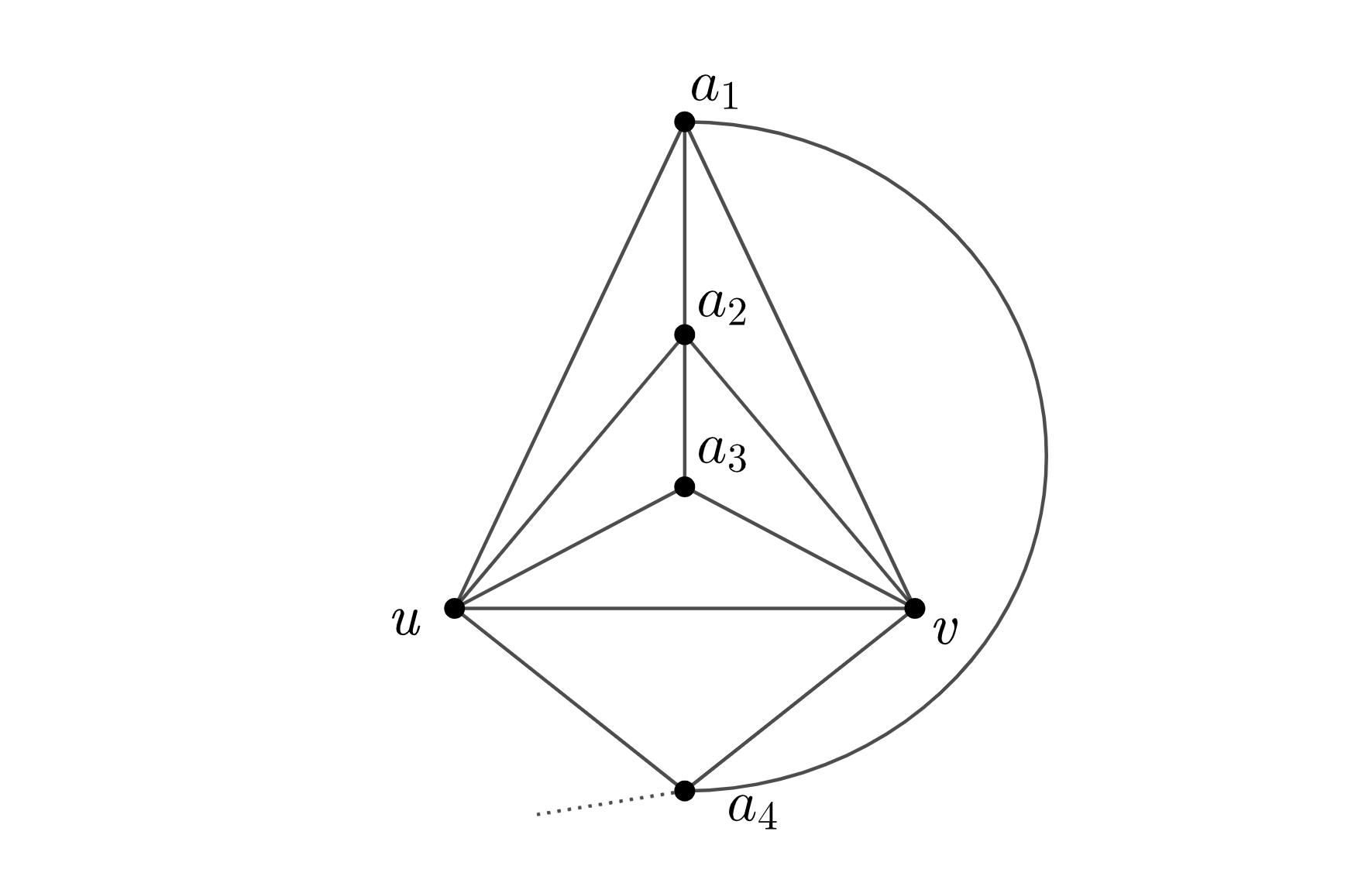}
  \caption{The star-block $B$ is a $5$-$5$ edge with $e(H_{1})=3$. }
  \label{fig_55_3}
\end{figure}

Assume that $d(a_{4})=4$.
Let $a_{4}a'_{4}\in E(G)$ and $B'=B+a'_{4}$. Note that $d(a'_{4})=3$, otherwise we find an $S_{3,3}$. If $d(a_{2})=4$ and $d(a_{3})=3$,  only $a_{1}$ can have a neighbor outside of $B'$.
When $d_{\mathcal{B}'}(a'_{4})=3$, then $w(B')\leq 13+\frac{3}{2}+\frac{1}{2}+1< \frac{5}{2}\cdot 7-1$.
When $d_{\mathcal{B}'}(a'_{4})=2$, then $w(B')\leq 13+\frac{3}{2}+\frac{1}{2}= \frac{5}{2}\cdot 7-\frac{5}{2}$.
When $d_{\mathcal{B}'}(a'_{4})=1$, it follows $w(B')\leq 13+\frac{3}{2}= \frac{5}{2}\cdot 7-3$.
If $d(a_{2})=5$ and $d(a_{3})=3$, there exists an edge $a_{2}a'_{2}\in E(G)$ with $d(a'_{2})=3$. Let $B^{*}=B'+a'_{2}$ and $\mathcal{B}^{*}$ is the corresponding refinement of $\mathcal{B}'$. 
When $d_{\mathcal{B}^{*}}(a'_{4})=3$ or $d_{\mathcal{B}^{*}}(a'_{2})=3$, we get $w(B^{*})\leq 14+\frac{5}{2}+\frac{2}{2}+1< \frac{5}{2}\cdot 8-1$.
When $d_{\mathcal{B}^{*}}(a'_{4}), d_{\mathcal{B}^{*}}(a'_{2})\leq 2$, then $w(B^{*})\leq 14+\frac{5}{2}+\frac{2}{2}=\frac{5}{2}\cdot 8-\frac{5}{2}$. 
when $d_{\mathcal{B}^{*}}(a'_{4})=d_{\mathcal{B}^{*}}(a'_{2})=1$, we have $w(B^{*})\leq 14+\frac{5}{2}= \frac{5}{2}\cdot 8-\frac{7}{2}$. 
If $d(a_{3})=4$, we have an edge $a_{3}a'_{3}$ with $d(a'_{3})=3$. Let $B^{*}=B'+a'_{3}$. The subsequent discussion is analogous, so it will not be reiterated here. We can prove that $B^{*}$ or some other star-block refined from $B^{*}$ satisfies the upper bound. 

Hence we have $d(a_{3})=d(a_{4})=3$ by the symmetry of these two vertices.

Next we show that $d(a_{1})\leq 4$ or $d(a_{2})\leq 4$. If $d(a_{1})=d(a_{2})=5$, there exist edges $a_{1}a'_{1}, a_{2}a'_{2}$ in $G$ and $d(a'_{1}),d(a'_{2})\leq 3$. If $a'_{1}\neq a'_{2}$, let $B'=B+a'_{1}+a'_{2}$. Then $w(B')\leq 14+\frac{4}{2}+\frac{2}{2}+\mathbf{1}_{B'}= \frac{5}{2}\cdot 8-3+\mathbf{1}_{B'}$, satisfying the upper bound. If $a'_{1}=a'_{2}$, then let $B'=B+a'_{1}$. We have $d_{\mathcal{B'}}(a'_{1})\leq 2$ and  $w(B')\leq 14+\frac{1}{2}+\frac{1}{2}=\frac{5}{2}\cdot 7-\frac{5}{2}$.  Specially, if there is no shared vertex in $G$, we have $w(B')\leq 14+\frac{1}{2}= \frac{5}{2}\cdot 7-3$.

Now we have $d(a_{1})\leq 4$ or $d(a_{2})\leq 4$, thus $w(B)\leq 12+\frac{1}{2}=\frac{5}{2}\cdot 6-\frac{5}{2}$.

\noindent\textbf{Case 5.2.} There are $3$ triangles sitting on $uv$.

Let $a_{1}$, $a_{2}$ and $a_{3}$ be the vertices adjacent to both $u$ and $v$. Let $b_{1}$ be the vertex only adjacent to $u$ and $b_{2}$ be the vertex only adjacent to $v$, see Figure~\ref{fig_55_b_1}$(a)$ as an example. Let $S=\{u, v, a_{1}, a_{2}, a_{3}, b_{1}, b_{2}\}$, $S_{1}=\{a_{1}, a_{2}, a_{3}\}$, $S_{2}=\{b_{1}, b_{2}\}$ and  $H_{i}=G[S_{i}]$ for $i\in \{1, 2\}$.

\begin{figure}[ht]
  \centering  \includegraphics[width=1\textwidth]{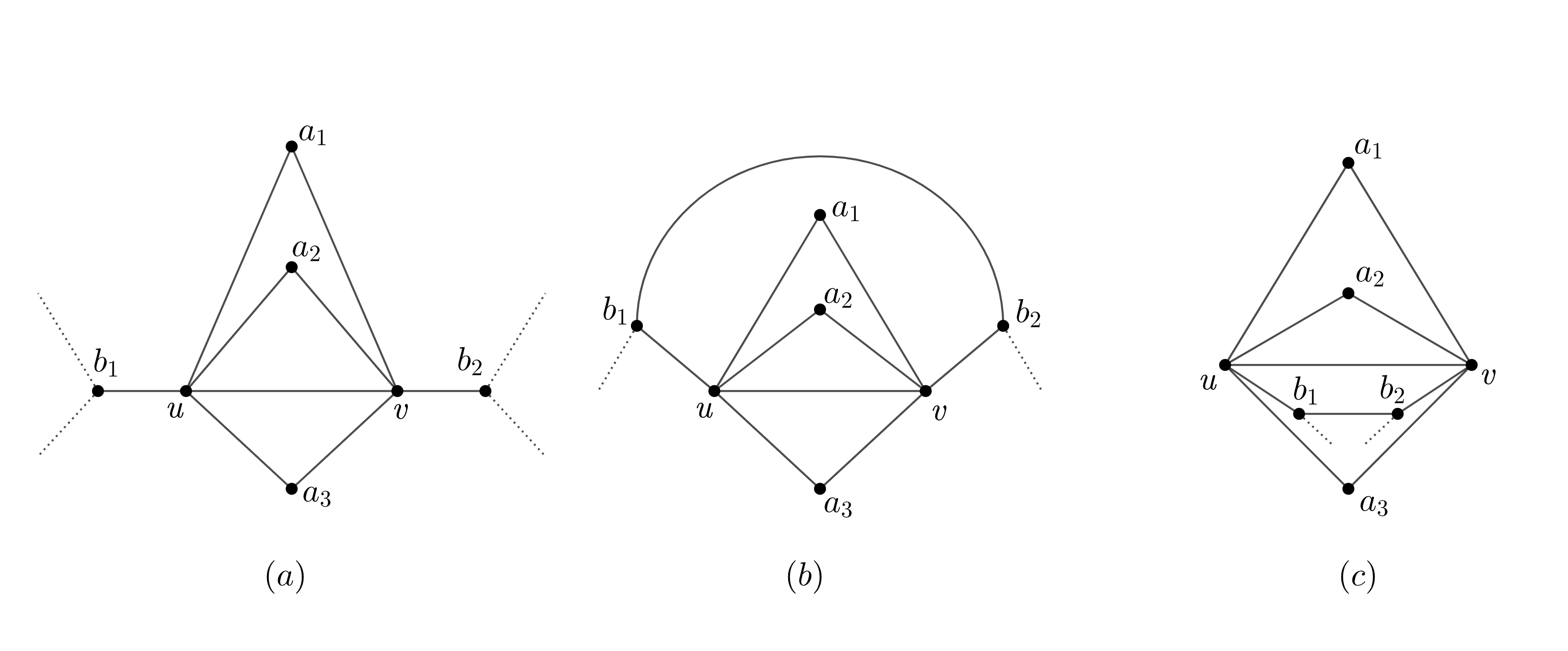}
  \caption{The star-block $B$ is a $5$-$5$ edge on which three triangles sitting. }
  \label{fig_55_b_1}
\end{figure}

Recall that each vertex in $S_{1}$ can have at most one neighbor outside of $B$. And it can be checked that $d(b_{1}), d(b_{2})\leq 4$. If $d(b_{1})=d(b_{2})=4$, then $b_{1}, b_{2}$ both have two neighbors in $S_{1}$. 

We may assume that there does not exist vertex of degree $6$ in $B$ by the cases discussed above. Let $s_{0}$ denote the number of shared vertices in $S_{1}$.

\noindent\textbf{Case 5.2.1.} $d(b_{1}), d(b_{2})\leq 3$.

\noindent\textbf{(I).} $d_{\mathcal{B}}(b_{1})=3, d_{\mathcal{B}}(b_{2})\leq 3$.

We have $d(b_{1})=3$ and $v(B)=7$. Note that $b_{1}$ has no neighbor in $S_{1}$. If $b_{2}a_{1}\in E(G)$, then $a_{1}$ can not have a neighbor outside of $B$. So it is determined that $d(a_{1})\leq 5$ and $d(a_{2}), d(a_{3})\leq 4$. Moreover any shared vertex in $S_{1}$ has degree $3$.

It is obtained $w(B)=\frac{1}{2}\sum\limits_{v\in B}d(v)+\frac{s}{2}+\mathbf{1}_{B}\leq \frac{1}{2}(5\cdot 3+4\cdot 2-s_{0}+3\cdot 2)+\frac{2+s_{0}}{2}+1=\frac{5}{2}\cdot 7-1$.

\noindent\textbf{(II).} $d_{\mathcal{B}}(b_{1})=d_{\mathcal{B}}(b_{2})=2$.

Assume that $b_{1}b_{2}\in E(G)$. Since $b_{1}, b_{2}$ are shared vertices, it is easy to know that $b_{1}, b_{2}$ is not adjacent to any vertex in $S_{1}$ and $d(b_{1})=d(b_{2})=3$.

There are two possible planar embeddings.

For the first planar embedding $(\alpha)$, as shown in Figure~\ref{fig_55_b_1}$(b)$, we have $a_{1}a_{3}\notin E(G)$, $d(a_{1}), d(a_{3})\leq 4$ and $d(a_{2})\leq 3$. When $s_{0}\leq 1$, $w(B)\leq \frac{1}{2}(5\cdot 2+4\cdot 2+3\cdot 3)+\frac{s_{0}+2}{2}\leq \frac{30}{2}=\frac{5}{2}\cdot 7-\frac{5}{2}$. When $s_{0}\geq 2$, $w(B)\leq \frac{1}{2}(5\cdot 2+4\cdot 2+3\cdot 3-s_{0}+1)+\frac{s_{0}+2}{2}= \frac{5}{2}\cdot 7-\frac{5}{2}$.

For the second planar embedding $(\beta)$, as shown in Figure~\ref{fig_55_b_1}$(c)$, there must exist a vertex in $S_{1}$, say $a_{2}$, such that the vertex $a_{2}$ is in a different region from vertices $b_{1}$ and $b_{2}$. If $d(a_{1})=5$, there exists an edge $a_{1}a'_{1}\in E(G)$ such that $d(a'_{1})\leq 3$, otherwise an $S_{3,3}$ is found. Let $B'=B+a'_{1}$ and $\mathcal{B}'$ is the corresponding refinement of $\mathcal{B}$. It is obtained that $w(B')\leq \frac{1}{2}(5\cdot 3+4\cdot 2-s_{0}+3\cdot 3)+\frac{3+s_{0}}{2}+\mathbf{1}_{B'}=\frac{5}{2}\cdot 8-\frac{5}{2}+\mathbf{1}_{B'}$, which satisfies the upper bound. Specially, if $d_{\mathcal{B}'}(a'_{1})=1$, we have $w(B')\leq 17=\frac{5}{2}\cdot 8-3$.

Therefore we can conclude that $b_{1}b_{2}\notin E(G)$. Next we will show that $d(a_{2}), d(a_{3})\leq 4$. In fact, if there exists a vertex of degree $5$, say $a_{3}$, then $a_{3}b_{1}, a_{3}b_{2}, a_{1}a_{3}\in E(G)$, which implies $d(a_{2})\leq 4$ and $a_{1}, a_{3}$ are not shared vertices. If $d(a_{1})=5$, then $a_{1}a_{2}\in E(G)$, as shown in Figure~\ref{fig_55_b_2}$(a)$. There is an $S_{3,3}$. So we get $d(a_{1})\leq 4$. 
If $d(a_{2})=4$, there is an edge $a_{2}a'_{2}$ with $d(a'_{2})\leq 3$. Let $B'=B+a'_{2}$ and $\mathcal{B}'$ is the corresponding refinement of $\mathcal{B}$. We have $w(B')\leq \frac{1}{2}(5\cdot 3+4\cdot 2+3\cdot 3)+\frac{2}{2}+\mathbf{1}_{B'}=\frac{5}{2}\cdot 8-3+\mathbf{1}_{B'}$, satisfying the bound.
Thus $d(a_{1})\leq 4$ and $d(a_{2})\leq 3$. It follows $w(B)\leq \frac{1}{2}(5\cdot 3+4\cdot 1-s_{0}+3\cdot 3)+\frac{2+s_{0}}{2}=  \frac{5}{2}\cdot 7-\frac{5}{2}$.

\begin{figure}[ht]
  \centering  \includegraphics[width=1\textwidth]{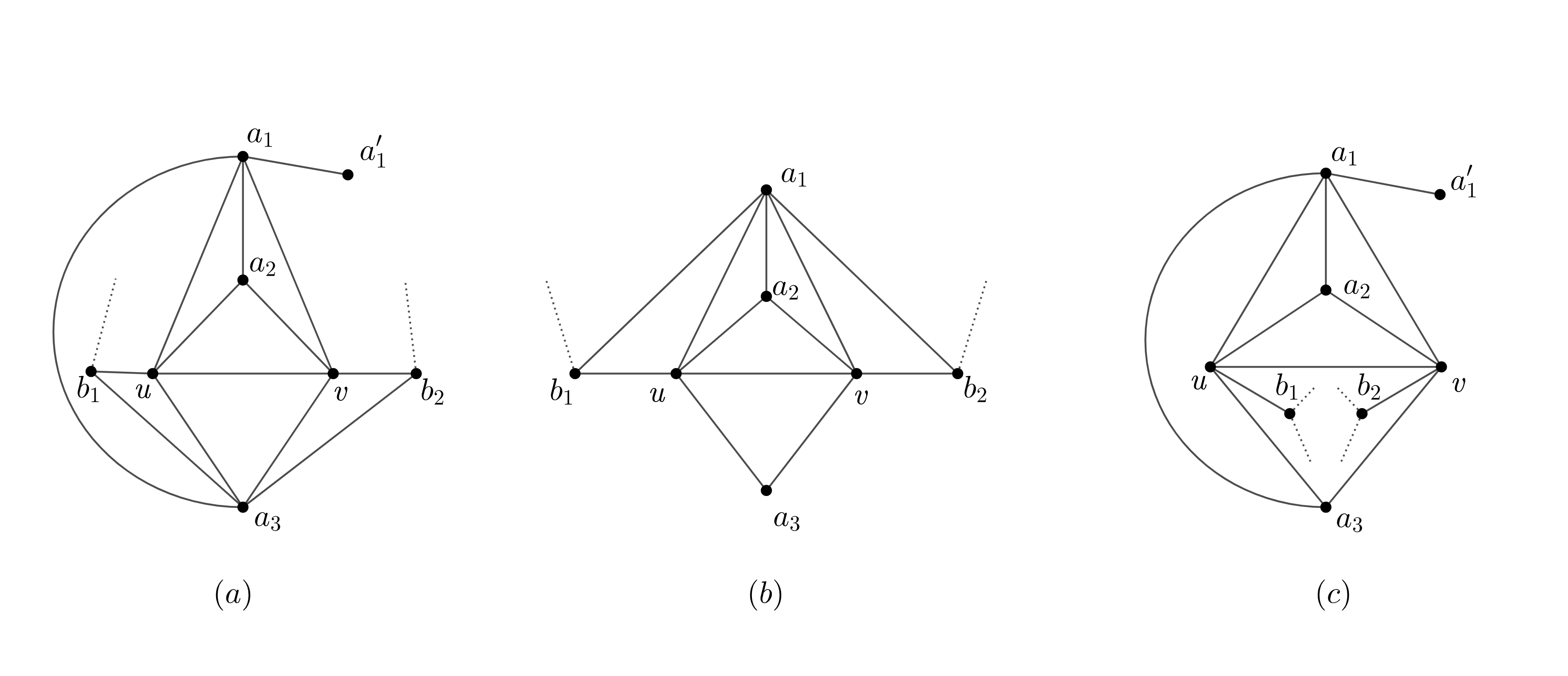}
  \caption{The star-block $B$ is a $5$-$5$ edge with $b_{1}b_{2}\notin E(G)$. }
  \label{fig_55_b_2}
\end{figure}

Now we may assume that $d(a_{2}), d(a_{3})\leq 4$ by the symmetry of these two vertices. If $b_{1}, b_{2}$ are in different regions, we obtain $t\geq 3$ and $w(B)\leq \frac{1}{2}(5\cdot 3+4\cdot 2-s_{0}+3\cdot 2)+\frac{s_{0}+2}{2}= \frac{31}{2}< \frac{5}{2}\cdot 7-\frac{5}{t}$.

We assume that $b_{1}, b_{2}$ are in the same region. Then there exists a vertex, say $a_{2}$, located in a different region. If $d(a_{2})=4$, then  let $a_{2}a'_{2}\in E(G)$ and $B'=B+a'_{2}$. It follows that $w(B')\leq \frac{1}{2}(5\cdot 3+4\cdot 2-s_{0}+3\cdot 3)+\frac{s_{0}+3}{2}+\mathbf{1}_{B'}=\frac{5}{2}\cdot 8-\frac{5}{2}+\mathbf{1}_{B'}$, which satisfies the upper bound. Specially, if $d_{\mathcal{B}'}(a'_{2})=1$, $w(B')\leq \frac{5}{2}\cdot 8-3$.

Hence we have $d(a_{2})\leq 3$. If $d(a_{1})=5$, there are only two possible planar embeddings, shown in Figure~\ref{fig_55_b_2}$(b,c)$. For the first planar embedding $(\alpha)$, it is known that $d(a_{3})\leq 3$ and $s_{0}\leq 1$. It follows $w(B)\leq \frac{1}{2}(5\cdot 3+3\cdot 4)+\frac{s_{0}+2}{2}\leq \frac{5}{2}\cdot 7-\frac{5}{2}$. For the second planar embedding $(\beta)$, we know $s_{0}=0$. Then $w(B)\leq \frac{1}{2}(5\cdot 3+4+3\cdot 3)+\frac{2}{2}\leq \frac{5}{2}\cdot 7-\frac{5}{2}$.

Now we have $d(a_{1}), d(a_{3})\leq 4$ and $d(a_{2})\leq 3$. Hence $w(B)\leq \frac{1}{2}[5\cdot 2+4\cdot 2+3\cdot 3-\mathbf{1}_{s_{0}}\cdot(s_{0}-1) ]+\frac{s_{0}+2}{2}\leq \frac{30}{2}=\frac{5}{2}\cdot 7-\frac{5}{2}$, where $\mathbf{1}_{s_{0}}$ is the characteristic function.

\noindent\textbf{(III).} $d_{\mathcal{B}}(b_{1})=2, d_{\mathcal{B}}(b_{2})=1$.

We demonstrate that the proof here is essentially the same as the aforementioned discussion, and there is a relationship between the weights of them.

Note that the calculations for $w(B)$ above were based on the assumption that all vertices in $B$ have degree at least $3$. As shown in Figure~\ref{fig_55_b_1}$(a)$, if $d(b_{2})\leq 2$, then the star-block structure here is a subgraph of a certain subcase, say $B'$, in (II). So it is a process that ensures a decrease in weight. We have $w(B)\leq w(B')-\frac{1}{2}$. By verifying the above results sequentially, we obtain $w(B)$ satisfies the upper bound.

Hence we can assume that $d(b_{2})=3$. If $b_{2}$ has a neighbor outside of $B$, it is checked that the subgraph structure is the same as previously discussed in (II). Note that $b_{2}$ is not a shared vertex anymore. Comparing to the previous value, the $w(B)$ here will be reduced by $1/2$. This means $w(B)$ satisfies the upper bound.

Thus $N(b_{2})\subset V(B)$. Note that $b_{1}$  has at least one neighbor outside by $d_{\mathcal{B}}(b_{1})=2$. There must exist a vertex in $S_{1}$, say $a_{1}$, which is adjacent to $b_{2}$ and not adjacent to $b_{1}$. Note that $a_{1}$ can not have a neighbor outside of $B$.

Now we make some modifications to this star-block. Let us cut the edge $a_{1}b_{2}$. And suppose that $a_{1}$ has a neighbor outside of $B$ and $b_{2}$ has a neighbor outside too. The modified star-block is denoted as $B'$. If $d_{\mathcal{B}'}(b_{2})=2$, $B'$ is a subcase in (II) and $w_{0}(B)=w_{0}(B')$. The difference between $B$ and $B'$ is the number of shared vertices they contain. It can be checked that $w(B)\leq w(B')-\frac{1}{2}$.

\noindent\textbf{(IV).} $d_{\mathcal{B}}(b_{1})=d_{\mathcal{B}}(b_{2})=1$.

As in the previous discussion, we can categorize the situation here as one of the aforementioned subcases in (III). 

If $b_{1}$ has a neighbor outside, it is easy to see that the star-block here corresponds to a certain subgraph in (III).

Assume that $a_{1}b_{1}$ is an edge. We make some modifications to $B$ by cutting the edge and keeping the other connections of $a_{1}$. Now suppose $b_{1}$ have a neighbor outside and $d_{\mathcal{B}'}(b_{1})=2$, and let $B'$ denote the new star-block.  It is known that $B'$ is a star-block in (III) and $w(B)=w(B')$.

\noindent\textbf{Case 5.2.2.}
$d(b_{1})=4, d(b_{2})\leq 3$.

Since $d(b_{1})=4$, $b_{1}$ has two neighbors, say $a_{1}, a_{3}$, in $S_{1}$, as shown in Figure~\ref{fig_55_c_1}$(a)$. Then $a_{1}, a_{3}, b_{1}$ are not shared vertices and $a_{2}, b_{1}$ are in different regions. This means that $d(a_{2})\leq 4$ and $a_{1}, a_{3}$ have no neighbor outside of $B$. Let $s_{0}$ be the number of shared vertices in $S_{1}$. Then $s_{0}\leq 1$.

\begin{figure}[ht]
  \centering  \includegraphics[width=0.95\textwidth]{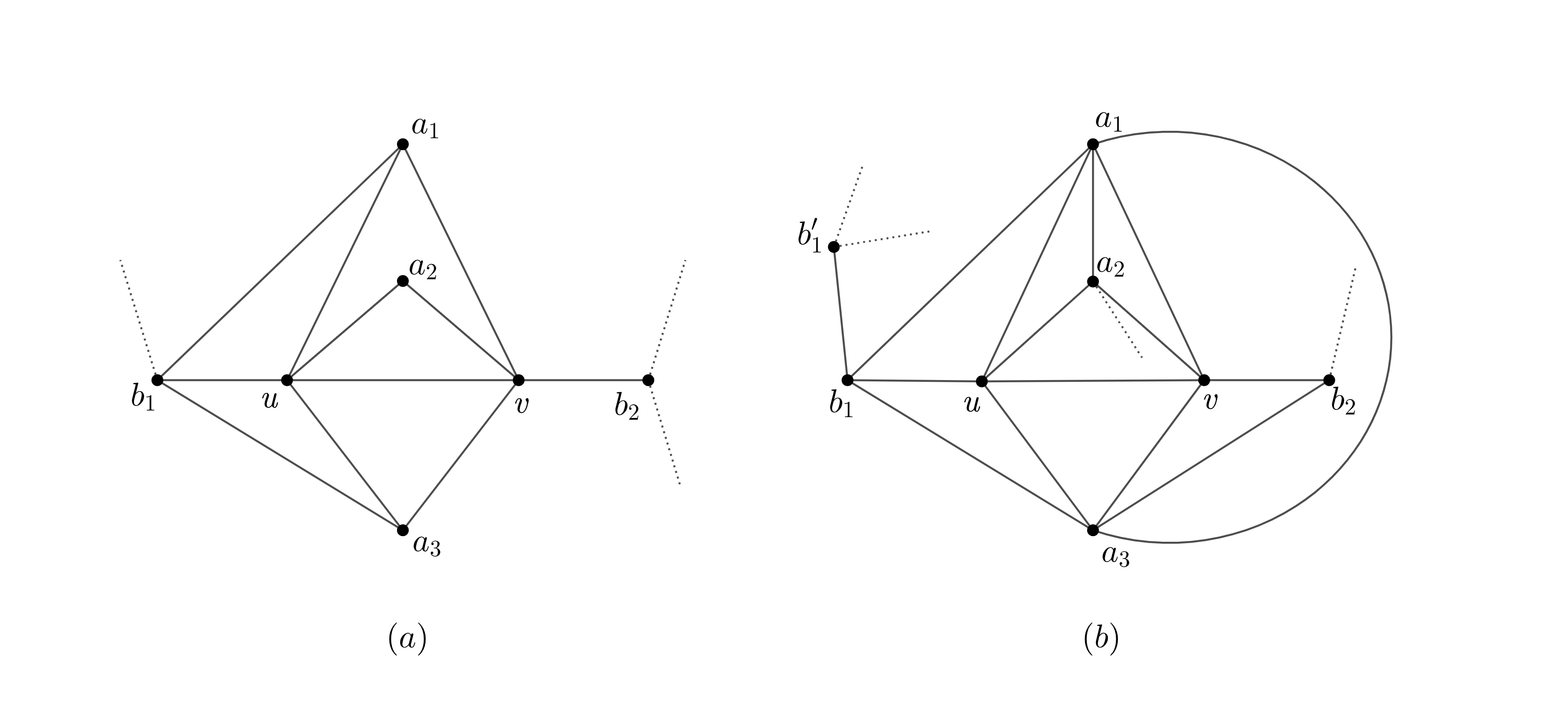}
  \caption{A $5$-$5$ edge with $d(b_{1})=4, d(b_{2})\leq 3$. }
  \label{fig_55_c_1}
\end{figure}

\noindent\textbf{(I).} $d_{\mathcal{B}}(b_{2})=3$.

Note that $b_{2}$ is not adjacent to any vertex in $S_{1}\cup \{b_{1}\}$. We have $d(a_{3})\leq 4$. Thus $w(B)\leq \frac{1}{2}(5\cdot 3+4\cdot 3-s_{0}+3\cdot 1)+\frac{s_{0}+1}{2}+\mathbf{1}_{B}= \frac{5}{2}\cdot 7-1$.

\noindent\textbf{(II).} $d_{\mathcal{B}}(b_{2})=2$.



If $a_{1}a_{3}\in E(G)$, $b_{1}, b_{2}$ are in different regions. Then there exists an edge $b_{1}b'_{1}$ such that $d(b'_{1})\leq 3$. Let $B'=B+b'_{1}$.

When $d_{\mathcal{B}'}(b'_{1})=3$, we have $w(B')\leq \frac{1}{2}(5\cdot 4+4\cdot 2-s_{0}+3\cdot 2)+\frac{s_{0}+2}{2}+\mathbf{1}_{B'}\leq \frac{38}{2}= \frac{5}{2}\cdot 8-1$.

When $d_{\mathcal{B}'}(b'_{1})=2$, we obtain $t\geq 3$. Then $w(B')\leq \frac{36}{2}\leq \frac{5}{2}\cdot 8-\frac{5}{t}$.

When $d_{\mathcal{B}'}(b'_{1})=1$, it follows $w(B')\leq \frac{1}{2}(5\cdot 4+4\cdot 2-s_{0}+3\cdot 2)+\frac{s_{0}+1}{2}\leq \frac{35}{2}=\frac{5}{2}\cdot 8-\frac{5}{2}$.

Now we may assume that $a_{1}a_{3}\notin E(G)$. It is easy to see that $d(a_{3})\leq 4$. Recall that $b_{2}$ has a neighbor outside of $B$. This implies that $d(a_{1})+d(a_{3})\leq 8$. So $w(B)\leq \frac{1}{2}(5\cdot 2+8+4\cdot 2-s_{0}+3\cdot 1)+\frac{s_{0}+1}{2}=\frac{5}{2}\cdot 7-\frac{5}{2}$.

\noindent\textbf{(III).} $d_{\mathcal{B}}(b_{2})=1$.

We have $w(B)\leq \frac{1}{2}(5\cdot 4+4\cdot 2-s_{0}+3)+\frac{s_{0}}{2}=\frac{5}{2}\cdot 7-2$. The equality holds when $d(b_{2})=3$, $d(a_{1})=d(a_{3})=5$ and $d(a_{2})=4$, otherwise we are done. It can be checked that $a_{1}$ is adjacent to $a_{3}$ as shown in~\ref{fig_55_c_1}$(b)$, which implies $b_{1}$ and $b_{2}$ are in different regions. So there exists an edge $b_{1}b'_{1}$ such that $d(b'_{1})\leq 3$. Let $B'=B+b'_{1}$. Then $w(B')\leq \frac{1}{2}(5\cdot 4+4\cdot 2-s_{0}+3\cdot 2)+\frac{s_{0}+1}{2}+\mathbf{1}_{B'}=\frac{5}{2}\cdot 8-\frac{5}{2}+\mathbf{1}_{B'}$, which satisfies the upper bound.

\noindent\textbf{Case 5.2.3.}
$d(b_{1})=d(b_{2})=4$.

It is easy to see that $d_{\mathcal{B}}(b_{1})=d_{\mathcal{B}}(b_{2})=1$ and $s_{0}\leq 1$. Note that $b_{1}, b_{2}$ both have two neighbors in $S_{1}$. Without loss of generality, we assume that $a_{1}, a_{3}$ are neighbors of $b_{1}$. Thus $d(a_{2})\leq 4$ and $a_{1}, a_{3}$ have no neighbor outside of $B$.

\noindent\textbf{(I).} $N(b_{1})=N(b_{2})$.

The graph of this situation is shown in Figure~\ref{fig_55_d_1}$(a)$. If $b_{1}b_{2}\in E(G)$, then $a_{1}a_{3}\notin E(G)$, which implies $d(a_{3})\leq 4$. If $d(a_{2})\leq 3$, we have $w(B)\leq \frac{1}{2}(5\cdot 3+4\cdot 3-s_{0}+3\cdot 1)+\frac{s_{0}}{2}= \frac{5}{2}\cdot 7-\frac{5}{2}$. If $d(a_{2})=4$, there exist edges $a_{1}a_{2}, a_{2}a'_{2}\in E(G)$ with $d(a'_{2})\leq 3$. Note that $s_{0}=0$. Let $B'=B+a'_{2}$ and $\mathcal{B'}$ be the corresponding refinement of $\mathcal{B}$. Hence $w(B')\leq \frac{1}{2}(5\cdot 3+4\cdot 4+3\cdot 1)+\frac{1}{2}+\mathbf{1}_{B'}=\frac{5}{2}\cdot 8-\frac{5}{2}+\mathbf{1}_{B'}$, which satisfies the upper bound.

Now we assume that $b_{1}b_{2}\notin E(G)$. There exist edges $b_{1}b'_{1}, b_{2}b'_{2}$ with $d(b'_{1}), d(b'_{2})\leq 3$. If $b'_{1}=b'_{2}$, then $a_{1}a_{3}\notin E(G)$, which implies $d(a_{3})\leq 4$. Let $B'=B+b'_{1}$. Note that $d_{\mathcal{B}'}(b'_{1})\leq 2$. Thus $w(B')\leq \frac{1}{2}(5\cdot 3+4\cdot 4-s_{0}+3\cdot 1)+\frac{s_{0}+1}{2}= \frac{5}{2}\cdot 8-\frac{5}{2}$. If $b'_{1}\neq b'_{2}$, let $B'=B+b'_{1}+b'_{2}$. Now we show that $d(a_{1})+d(a_{2})+d(a_{3})\leq 13$. In fact, if $a_{1}a_{3}\in E(G)$, then $d(a_{1}), d(a_{3})\geq 5$. Since we assume that there does not exist a star-block like $6$-$5$ edge, we get $d(a_{1})=d(a_{3})=5$. Then $a_{1}a_{2}\notin E(G)$, which implies $d(a_{2})\leq 3$. If $a_{1}a_{3}\notin E(G)$, it is easy to know $d(a_{1})\leq 5$, $d(a_{2})\leq 4$, $d(a_{1})\leq 4$. Thus it is obtained $d(a_{1})+d(a_{2})+d(a_{3})\leq 13$. Therefore we have $w(B')\leq \frac{1}{2}(5\cdot 2+13+4\cdot 2+3\cdot 2)+\frac{s_{0}+2}{2}+\mathbf{1}_{B'}\leq \frac{5}{2}\cdot 9-\frac{5}{2}+\mathbf{1}_{B'}$, which satisfies the upper bound.

\begin{figure}[ht]
  \centering  \includegraphics[width=0.95\textwidth]{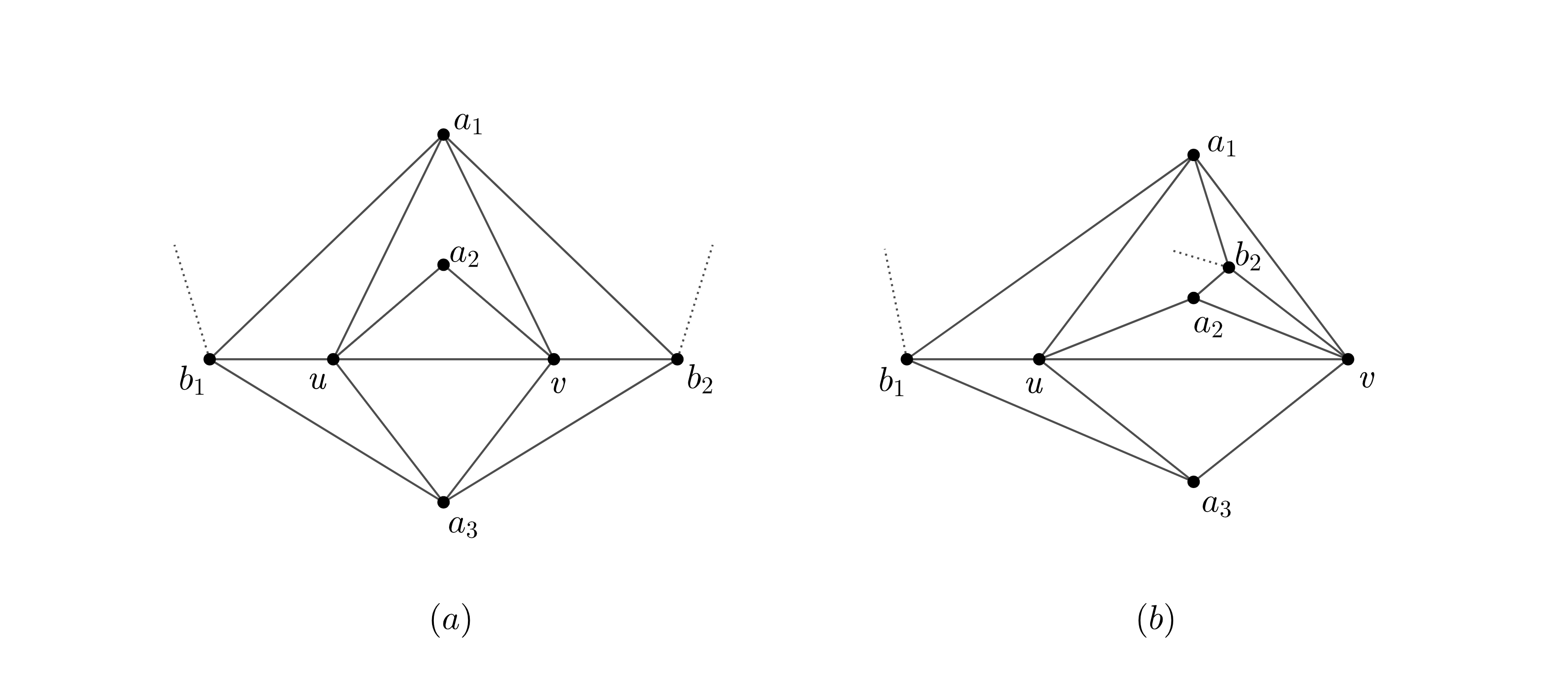}
  \caption{A $5$-$5$ edge with $d(b_{1})=d(b_{2})=4$. }
  \label{fig_55_d_1}
\end{figure}

\noindent\textbf{(II).} $N(b_{1})\neq N(b_{2})$.

Assume that $\{a_{1}, a_{2}\} \subseteq N(b_{2})$, as shown in Figure~\ref{fig_55_d_1}$(b)$. Obviously, $b_{1}, b_{2}$ are in different regions and any vertex in $S_{1}$ has no neighbor outside of $B$. That means $s_{0}=0$. By the assumption that there is no star-block like $6$-$5$ edge, so $a_{1}a_{2}, a_{1}a_{3}$ can not both be edges in $G$. It is obtained $d(a_{1})+d(a_{2})+d(a_{3})\leq 12$. Let $b_{1}b'_{1},b_{2}b'_{2} \in E(G)$ and $B'=B+b'_{1}+b'_{2}$, we have $w(B')\leq \frac{1}{2}(5\cdot 2+12+4\cdot 2+3\cdot 2)+\frac{2}{2}+\mathbf{1}_{B'}\leq \frac{5}{2}\cdot 9-\frac{5}{2}+\mathbf{1}_{B'}$, which satisfies the upper bound.

\noindent\textbf{Case 6.} $B$ is a $5$-$4$-$5$ path.

We may assume that there are no star-blocks like $6$-$5$ edge, $5$-$5$ edge. This means that all $5$-degree vertices form an independent set.  Let $u, v, w$ be the vertices in the $5$-$4$-$5$ path and $d(v)=4$. Consider the $5$-$4$ edge $uv$. Since $G$ is $S_{3,3}$-free, the number of triangles sitting on $uv$ is at least $2$. There are three possible subgraphs, seen in Figure~\ref{fig_545}. 

\begin{figure}[ht]
  \centering  \includegraphics[width=0.95\textwidth]{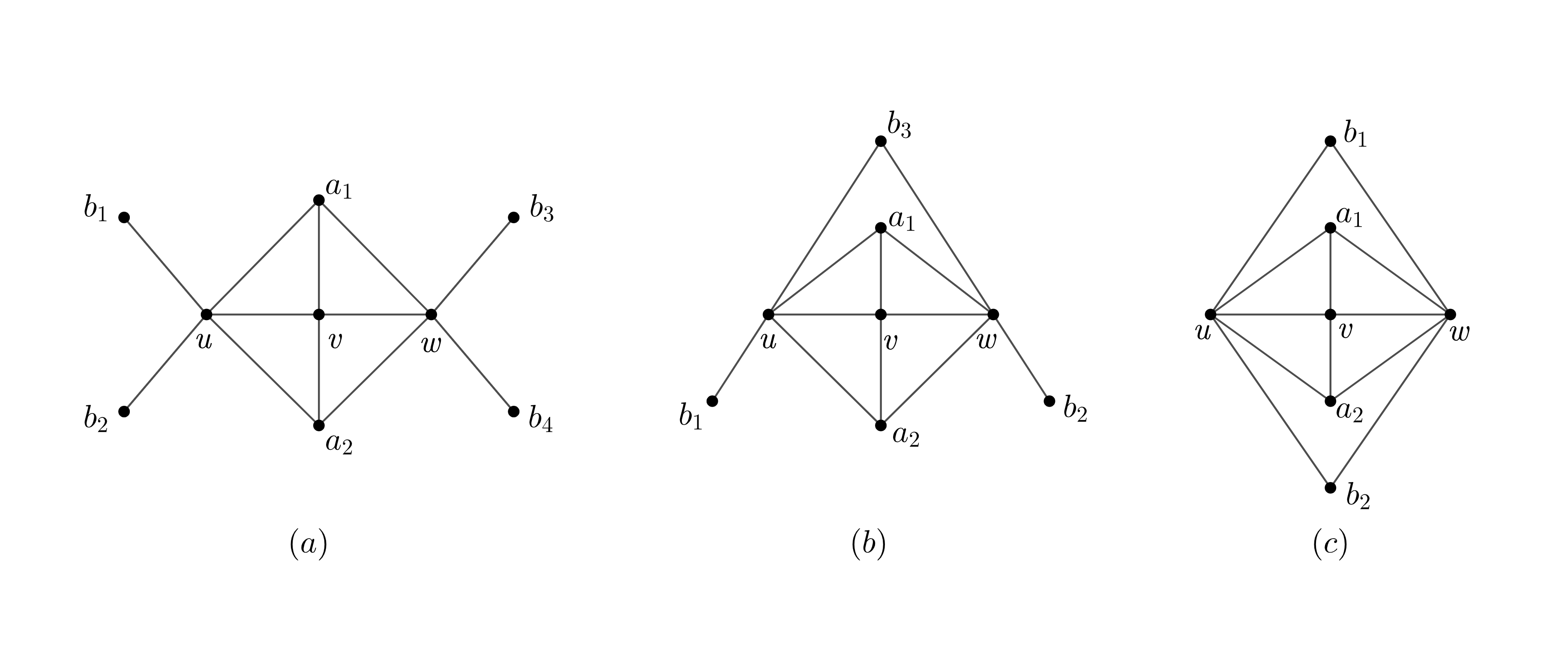}
  \caption{All $5$-$4$-$5$ paths. }
  \label{fig_545}
\end{figure}

Now we consider the subgraphs in turn.

For the subgraph $(a)$, let $S=\{u, v, w, a_{1}, a_{2}, b_{1}, b_{2}, b_{3}, b_{4}\}$, $S_{1}=\{a_{1}, a_{2}\}$ and $S_{2}=\{b_{1}, b_{2}, b_{3}, b_{4}\}$. It can be determined that each vertex in $S_{2}$ can have at most two neighbors in $S_{2}\cup G\backslash B$ and any vertex in $S_{1}$ has no neighbor in $S_{2}\cup G\backslash B$. Thus $d(b_{1}), d(b_{2}),d(b_{3}),d(b_{4})\leq 3$. It follows $w(B)\leq \frac{1}{2}(5\cdot 2+4\cdot 3+3\cdot 4)+\frac{4}{2}+\mathbf{1}_{B}=\frac{5}{2}\cdot 9-\frac{7}{2}+\mathbf{1}_{B}$. Specially, if there does not exist any shared vertex, we have $w(B)\leq \frac{5}{2}\cdot 9-\frac{11}{2}$.

Similarly, for the subgraph $(b)$,  we have $w(B)\leq \frac{1}{2}(5\cdot 2+4\cdot 4+3\cdot 2)+\frac{3}{2}+\mathbf{1}_{B}=\frac{5}{2}\cdot 8-\frac{5}{2}+\mathbf{1}_{B}$.

For the subgraph $(c)$, $b_{1}, b_{2}$ both have two neighbors in $B$. Hence, $d_{\mathcal{B}}(b_{1}), d_{\mathcal{B}}(b_{2})\leq 2$, which implies $\mathbf{1}_{B}=0$. If $d_{\mathcal{B}}(b_{1})=2$ or $d_{\mathcal{B}}(b_{2})=2$, then $w(B)\leq \frac{1}{2}(5\cdot 2+4\cdot 3+3\cdot 2)+\frac{2}{2}=\frac{5}{2}\cdot 7-\frac{5}{2}$. If $d_{\mathcal{B}}(b_{1})=d_{\mathcal{B}}(b_{2})=1$, we have $w(B)\leq \frac{1}{2}(5\cdot 2+4\cdot 5)=\frac{5}{2}\cdot 7-\frac{5}{2}$.

In summary, $w(B)$ satisfies the upper bound.

\noindent\textbf{Case 7.} $B$ is a $5$-$4^{-}$ star.

Let $u$ be the vertex of degree 5, we claim that any vertex $v\in N(u)$ can not be a shared vertex. In fact, there are at least $2$ triangles sitting on the edge $5$-$4$, otherwise an $S_{3,3}$ is found in $G$. This means that $v$ have at least three neighbors in $N[u]\backslash v$. By the definition of star-block, $v$ can not be a shared vertex.

Therefore we have $w(B)=\frac{1}{2}(5+5\cdot 4)=\frac{5}{2}\cdot 6 -\frac{5}{2}$.

\section{Proof of Lemma~\ref{lemma_eq1}}

Since $G$ contains only one star-block $B$, then all vertices outside of the star-block have degree at most $4$. There does not exist any shared vertex, which implies $w(B)=w_{0}(B)$.

If $G$ is disconnected, we assume that $G_{1}$ is the component containing the star-block and $G_{2}=G\backslash G_{1}$. Let $k_{1}=v(G_{1})$, $k'_{1}=v(B)$ and $k_{2}=v(G_{2})$. By Lemma~\ref{lemma_w}, $w(B)\leq \frac{5}{2}v(B)-\frac{5}{2}$. It follows that $e(G_{1})=w(G_{1})=w(B)+w(G_{1}\backslash B)\leq \frac{5}{2}k'_{1}-\frac{5}{2}+\frac{1}{2}[4(k_{1}-k'_{1})]$. When $1\leq k_{2}\leq 2$, then $e(G_{2})\leq \frac{5}{2}k_{2}-\frac{5}{2}$. When $3\leq k_{2}\leq 7$, it is easy to check that $e(G_{2})\leq 3k_{2}-6\leq \frac{5}{2}k_{2}-\frac{5}{2}$. When $k_{2}\geq 8$, we also have $e(G_{2})=\frac{1}{2}\sum\limits_{v\in V(G_{2})}d(v)\leq \frac{4k_{2}}{2}\leq \frac{5}{2}k_{2}-\frac{5}{2}$.

Hence, we have 
\begin{align*}
    e(G)&=e(G_{1})+e(G_{2})\\
    &\leq \frac{5}{2}k'_{1}-\frac{5}{2}+2(k_{1}-k'_{1})+ \frac{5}{2}k_{2}-\frac{5}{2}\\
    &=\frac{5}{2}n-5-\frac{1}{2}(k_{1}-k'_{1})\\
    &\leq \frac{5}{2}n-5.
\end{align*}      

Now we can assume that $G$ is connected and we will discuss each type of star-blocks in turn. Let $k=v(B)$.

  
\noindent\textbf{Case 1.} $G$ contains a $5^{+}$-$3^{-}$ star.

Let $B$ be the $5^{+}$-$3^{-}$ star and  $k\geq 6$. It follows $w(B)\leq \frac{1}{2}[(k-1)+3(k-1)]=2(k-1)\leq \frac{5}{2}k-5$. 
Thus  
\begin{align*}
    e(G)&=w(B)+w(G\backslash B)\\
    &\leq \frac{5}{2}k-5+\sum\limits_{v\in G\backslash B}d(v)\\
    &\leq \frac{5}{2}k-5+\frac{1}{2}\cdot 4(n-k)\\
    &\leq \frac{5}{2}n-5.    
\end{align*}

\noindent\textbf{Case 2.} $G$ contains a $6$-$6$ edge, or $6$-$5$ edge, or $6$-$4$ edge.

Referring to the final discussion in the proof for Cases $2,3,4$ of Lemma~\ref{lemma_w}, we can deduce that $B$ is the elementary star-block on $7$ vertices. Furthermore, if $B$ is a connected component, then $n=k=7$ and $e(G)=w(B)\leq 15$. If there are other vertices not contained in this star-block, we have $w(B)\leq \frac{5}{2}\cdot 7-3$. Then
\begin{align*}
    e(G)&=w(B)+w(G\backslash B)\\
    &\leq \frac{5}{2}\cdot 7-3+\frac{1}{2}[4(n-7)]\\
    &= 2n+\frac{1}{2}.
\end{align*}    

Since $e(G)$ is an integer, we have $e(G)\leq 2n$. This means $e(G)\leq 16$ when $n=8$, $e(G)\leq 18$ when $n=9$ and $e(G)\leq 20$ when $n=10$. Furthermore, when $n\geq 11$, it follows $e(G)\leq 2n+\frac{1}{2}\leq \frac{5}{2}n-5$.

\noindent\textbf{Case 3.} $G$ contains a $5$-$5$ edge.

We will show that there must exist a star-block in $G$ whose weight is not large.

\begin{claim}
    For any type of $5$-$5$ edge except one certain subcase, there exists a star-block $B$ containing it such that $w(B)\leq \frac{5}{2}k-\frac{k-1}{2}=2k+\frac{1}{2}$. 
\end{claim}

\begin{proof}
    If there are $4$ triangles sitting on $uv$, as shown in Figure \ref{fig_55}$(a)$, it can be determined based on the proof for Case 5.1 of Lemma~\ref{lemma_w}.

    Then it remains to prove the case that there are $3$ triangles sitting on $uv$, as shown in Figure~\ref{fig_55_b_1}$(a)$. However, there are too many subcases  and it is tedious to check each one individually. Here, we provide a new proof.
    
    Let $uv$ be the $5$-$5$ edge in $G$. Let $S=\{u, v, a_{1}, a_{2}, a_{3}, b_{1}, b_{2}\}$, $S_{1}=\{a_{1}, a_{2}, a_{3}\}$ and $H_{1}=G[S_{1}]$. Without loss of generality, we may assume that there does not exist vertex of degree  $6$ in $G$.
    
    \noindent\textbf{(I).} $d(b_{1}), d(b_{2})\leq 3$.
    
    Note that $a_{2}, a_{3}$ can not both be the vertices of degree $5$. In fact, if $d(a_{2})=5$, then $a_{2}b_{1}, a_{2}b_{2}\in E(G)$, which implies $d(a_{3})\leq 4$. Thus $d(a_{2})+d(a_{3})\leq 9$.

    Assume $d(a_{3})=5$. If $d(a_{2})\leq 3$, then $w(B)\leq \frac{1}{2}(5\cdot 4+3\cdot 3)=\frac{5}{2}\cdot 7-3$. If $d(a_{2})=4$, there exists an edge $a_{2}a'_{2}\in E(G)$. Let $B'=B+a'_{2}$. Hence $w(B')\leq \frac{1}{2}(5\cdot 4+4+3\cdot 3)=\frac{5}{2}\cdot 8-\frac{7}{2}$.

    Hence we may assume that $d(a_{3})\leq 4$ and $d(a_{2})\leq 4$. Then $w(B)\leq \frac{1}{2}(5\cdot 3+4\cdot 2+3\cdot 2)=\frac{5}{2}\cdot 7-3$. 

    \noindent\textbf{(II).} $d(b_{1})=4, d(b_{2})\leq 3$.
    
    Note that $b_{1}$ has two neighbors in $S_{1}$, as shown in Figure~\ref{fig_55_c_1}$(a)$. Assume that $a_{1}b_{1}, a_{3}b_{1}\in E(G)$. Then $a_{1}, a_{3}$ have no neighbor outside of $B$ and $b_{1}, a_{2}$ are in different regions.

    Since $d(b_{1})=4$, there exists an edge $b_{1}b'_{1}\in E(G)$ with $d(b'_{1})\leq 3$. 
    
    If $b'_{1}=b_{2}$, then $a_{1}a_{3}\notin E(G)$ and $b_{2}a_{1},b_{2}a_{3}$ can not both be edges, which implies $d(a_{1})+d(a_{3})\leq 8$. So $w(G)\leq \frac{1}{2}(5\cdot 2+4\cdot 2+3+8)= \frac{5}{2}\cdot 7-3$.
    
    If $b'_{1}\neq b_{2}$. If $d(a_{1})+d(a_{2})+d(a_{3})= 14$, then we can assume $d(a_{3})=5$ and $d(a_{2})=4$, so the graph is shown in~\ref{fig_55_c_1}$(b)$. Let $b_{1}b'_{1}, b_{2}b'_{2}\in E(G)$, and $B'=B+b'_{1}+b'_{2}$. We have $w(B')\leq \frac{1}{2}(5\cdot 4+4\cdot 2+3\cdot 3)=\frac{5}{2}\cdot 9-4$. Hence we may assume that $d(a_{1})+d(a_{2})+d(a_{3})\leq 13$. Let $B'=B+b'_{1}$. It is obtained that $w(B')\leq \frac{1}{2}(5\cdot 2+4+3\cdot 2+13)=\frac{5}{2}\cdot 8-\frac{7}{2}$.

    \noindent\textbf{(III).} $d(b_{1})=d(b_{2})=4$.
    
    There exist edges $b_{1}b'_{1}, b_{2}b'_{2}\in E(G)$, as shown in Figure~\ref{fig_55_d_1}.
    
    \noindent\textbf{(i).} $b'_{1}=b_{2}$ and $b'_{2}=b_{1}$.
    
    Note that $a_{1}a_{3}\notin E(G)$, which implies $d(a_{3})=4$.

    Assume that $d(a_{1})=5$. We have $a_{1}a_{2}\in E(G)$. If $d(a_{2})=4$, there exists an edge $a_{2}a'_{2}$. Then $a_{2}a'_{2}$ is a cut edge and  $V(G)\geq 8$. We will show later that if $G$ contains this subgraph, then $e(G)\leq \frac{5}{2}n-5$. If $d(a_{2})=3$. Then $B$ is a connected component on $7$ vertices with $15$ edges.

    Hence $d(a_{1})=4$. Then $a_{1}a_{2}\notin E(G)$, which implies $d(a_{2})\leq 3$. It follows $w(B)\leq \frac{29}{2}=\frac{5}{2}\cdot 7-3$.

    \noindent\textbf{(ii).} $b'_{1}=b'_{2}$.
    
    Similarly, we have $a_{1}a_{3}\notin E(G)$ and  $d(a_{3})=4$.

    If $d(a_{2})=4$, there exists an edge $a_{2}a'_{2}$ with $d(a'_{2})\leq 3$. Let $B=B+a'_{2}+b'_{1}$. It follows  $w(B')\leq \frac{1}{2}(5\cdot 3+4\cdot 4+3\cdot 2)=\frac{5}{2}\cdot 9-4$.

    If $d(a_{2})\leq 3$, let $B=B+b'_{1}$. It follows 
    $w(B')\leq \frac{1}{2}(5\cdot 3+4\cdot 3+3\cdot 2)=\frac{5}{2}\cdot 8-\frac{7}{2}$.
    
    \noindent\textbf{(iii).} $b'_{1}\neq b'_{2}$.
    
    There are two possible planar embeddings, as show in Figure~\ref{fig_55_d_1}.

    For the first planar embedding, $b_{1}, b_{2}$ are in a different region from $a_{2}$. We will show that $d(a_{1})+d(a_{2})+d(a_{3})\leq 13$. In fact, if $d(a_{3})=5$, then $a_{1}a_{3}\in E(G)$ and $a_{1}a_{2}\notin E(G)$, which means that $d(a_{2})\leq 3$.
    Let $B'=B+b'_{1}+b'_{2}$. Then $w(B')\leq \frac{1}{2}(5\cdot 2+4\cdot 2+3\cdot 2+13)=\frac{5}{2}\cdot 9-4$.

    For the second planar embedding, $a_{2}, b_{2}$ are in a different region from $b_{1}$. It is easy to check that $d(a_{1})\leq 5, d(a_{2})\leq 4, d(a_{3})\leq 4$.
    Let $B'=B+b'_{1}+b'_{2}$. Similarly we have $w(B')\leq \frac{37}{2}=\frac{5}{2}\cdot 9-4$.

    Now we prove that if $G$ contains such subgraph in (i), then $e(G)\leq \frac{5}{2}n-5$. When $n\geq 13$, we have $e(G)\leq w(B)+\frac{1}{2}\cdot 4(n-7)\leq 2n+\frac{3}{2}\leq \frac{5}{2}n-5$.

    Recall that $G$ is connected and $a_{2}a'_{2}$ is a cut edge. When $n\leq 12$, we may assume that there are two components $G_{1}, G_{2}$ connected by $a_{2}a'_{2}$. Then 
    \begin{align*}           
        e(G)&=e(G_{1})+e(G_{2})+1\\
        &=e(G_{2})+16\\
        &\leq 3(n-7)-6+16\\
        &=3n-11\\
        &\leq \frac{5}{2}n-5.
    \end{align*}

   Therefore if $G$ contains this certain type of $5$-$5$ edge, $e(G)\leq \frac{5}{2}n-5$.
\end{proof}

Hence, 
\begin{align*}
    e(G)&=w(B)+w(G\backslash B)\\
     &\leq \frac{5}{2}k-\frac{k-1}{2}+\frac{1}{2}\cdot 4(n-k)\\
    &\leq 2n+\frac{1}{2}.
\end{align*}

Similarly, we have $e(G)\leq 2n$ when $8\leq n\leq 10$ and $e(G)\leq \frac{5}{2}n-5$ when $n\geq 11$.

\noindent\textbf{Case 4.} $G$ contains a $5$-$4$-$5$ path or $5$-$4^{-}$ star.

Based on the proof in Lemma~\ref{lemma_w}, it is easy to see that $w(B)\leq \frac{5}{2}k-\frac{k-1}{2}$. Similarly, $e(G)$ satisfies the upper bound.

Note that $2n=\frac{5}{2}n-5$ when $n=10$. Therefore the lemma holds.

\section{Proof of Lemma~\ref{lemma_eq2} and Theorem~\ref{thm}}

In this section, we first prove the Lemma~\ref{lemma_eq2}, and then provide the proof of upper bound in theorem~\ref{thm}.

\begin{proof}
    
Note that if there is a vertex of degree at least $5$, then it must be contained in some star-block. Since $G$ has a star-block partition $G=G_{1}+G_{2}$, it follows that
$$e(G)=w_{0}(G_1)+w_{0}(G_{2})\leq w_{0}(G_1)+\frac{1}{2}\sum\limits_{v\in V(G_{2})}4=w_{0}(G_1)+2v(G_{2}).$$

It suffices to prove that
$$w_{0}(G_1)\leq \frac{5}{2}v(G_{1})-5.$$

Recall that $G_{1}$ has a star-block base $\mathcal{B}$. Let

\begin{itemize}
    \item $r_{1}\coloneqq \big|\{v\in V(G_{1}): d_{\mathcal{B}}(v)=2\ and\ d(v)\leq 3\}\big|$,
    
    \item $r_{2}\coloneqq \big|\{v\in V(G_{1}): d_{\mathcal{B}}(v)=2\ and\ d(v)=4\}\big|$,
    
    \item $r_{3}\coloneqq \big|\{v\in V(G_{1}): d_{\mathcal{B}}(v)=3\}\big|$.
\end{itemize}

Then we have
$$\sum\limits_{B\in \mathcal{B}}w_{0}(B)=w_{0}(G_{1})+\frac{3}{2}r_{1}+2r_{2}+3r_{3}.$$

Let $t_{0}=|\mathcal{B}_{0}|$, $t_{1}=|\mathcal{B}_{1}|$, $t_{2}=|\mathcal{B}_{2}|$. By Lemma~\ \ref{lemma_w}, we have
\begin{align*}
    \sum\limits_{B\in \mathcal{B}}w_{0}(B)&=\sum\limits_{B\in \mathcal{B}}(w(B)-\frac{s}{2}-\frac{s'}{4}-\mathbf{1}_{B})\\
    &=\sum\limits_{B\in \mathcal{B}}(w(B))-\sum\limits_{B\in \mathcal{B}}(\frac{s}{2}+\frac{s'}{4}+\mathbf{1}_{B})\\
    &\leq \frac{5}{2}\sum\limits_{B\in \mathcal{B}}v(B)-\frac{5}{2}t_{0}-\frac{5t_{1}}{t}-t_{2}-(r_{1}+\frac{1}{2}r_{2}+\frac{3}{2}r_{3}+t_{2})\\
    &=\frac{5}{2}(v(G_{1})+r_{1}+r_{2}+2r_{3})-\frac{5}{2}t_{0}-\frac{5t_{1}}{t}-t_{2}-(r_{1}+\frac{1}{2}r_{2}+\frac{3}{2}r_{3}+t_{2}).
\end{align*}

Combining the results, we conclude that
$$w_{0}(G_{1})\leq \frac{5}{2}v(G_{1})+(\frac{1}{2}r_{3}-t_{2})-\frac{5}{2}t_{0}-\frac{5t_{1}}{t}-t_{2}.$$

Next we will show that $r_{3}\leq 2t_{2}-4$ when $t_{2}\neq 0$. Recall that $r_{3}$ is the number of shared vertices in $G_{1}$ with $d_{\mathcal{B}}(v)=3$ and $t_{2}$ is the number of star-blocks in $\mathcal{B}$ containing such type of vertices. We construct an auxiliary bipartite graph $(X, Y)$ such that $|X|=r_{3}$ and $|Y|=t_{2}$. Each vertex in $X$ represents a vertex $v$ in $G_{1}$ with $d_{\mathcal{B}}(v)=3$ and each vertex in $Y$ represents a star-block in $\mathcal{B}_{2}$. Moreover, the edge $xy$ means that $x$ is contained in the star-block $y$ for $x\in X$ and $y\in Y$. It is easy to check that the auxiliary bipartite graph is a planar graph since $G$ is a planar graph and every vertex in $X$ has degree exactly $3$. So we have $3r_{3}\leq 2(r_{3}+t_{2})-4$. This implies $r_{3}\leq 2t_{2}-4$.

Hence we have
\begin{align*}
    w_{0}(G_{1})\leq \frac{5}{2}v(G_{1})-2\cdot \mathbf{1}_{t_{2}}-\frac{5}{2}t_{0}-\frac{5t_{1}}{t}-t_{2},
\end{align*}
where $\mathbf{1}_{t_{2}}$ is the characteristic function of $t_{2}$.

If $t_{2}>0$, then $t_{2}\geq 3$. It follows that $w_{0}(G_{1})\leq \frac{5}{2}v(G_{1})-5$. Assume that $t_{2}=0$. If $t_{1}>0$, then $w_{0}(G_{1})\leq \frac{5}{2}v(G_{1})-5$ since $t_{1}=t_{1}+t_{2}=t$. If $t_{1}=t_{2}=0$, then $w_{0}(G_{1})\leq \frac{5}{2}v(G_{1})-5$ since $t_{0}=|\mathcal{B}|\geq 2$.

It should be noted that we did not consider  vertices of degree 2 in our proof. In fact, the calculations for $w(B)$ were based on the assumption that all vertices in $B$ have degree at least $3$. Assume there exists a $2$-degree vertex $w$ in $G$. When $d_{\mathcal{B}}(w)=1$, it can be easily verified. When $d_{\mathcal{B}}(w)=2$, we show that this will not affect our proof either. 

Note that the core of the proof is the following inequality
$$w_{0}(G_{1})+\frac{3}{2}r_{1}+2r_{2}+3r_{3}\leq \sum\limits_{B\in \mathcal{B}}(w(B)-\frac{s}{2}-\frac{s'}{4}-\mathbf{1}_{B}).$$

If there exists a shared vertex $w$ of degree $2$, then the value of $w_{0}(G_{1})+\frac{3}{2}r_{1}+2r_{2}+3r_{3}$ would be reduced by $1/2$. However $w(B)$ would be reduced by $1/2$ if $w\in B$. And there are two star-blocks containing $w$. Thus $\sum\limits_{B\in \mathcal{B}}(w(B)-\frac{s}{2}-\frac{s'}{4}-\mathbf{1}_{B})$ would be reduced by at least $1$. Hence, ignoring each shared vertex of degree $2$ results in an increase of $\frac{1}{2}$ to the value on the left side of the inequality, and an increase of $1$ to the right side. This does not affect our calculation above.
\end{proof}

Now we give the proof of  Theorem~\ref{thm}.
\begin{proof}
    Given any planar graph $G$ on $n\leq 7$ vertices, there does not exist an $S_{3,3}$ obviously. Thus $e(G)\leq 3n-6$.

    By Lemma~\ref{lemma_eq1} and Lemma~\ref{lemma_eq2}, it remains to prove the theorem for $|\mathcal{B}|=0$. This means  there is no star-block contained in $G$.  Thus we have $\Delta(G)\leq 4$.
    Hence $e(G)=\sum\limits_{v\in V(G)}d(v)\leq 2n$.  
    When $n\geq 10$, we have $e(G)\leq \frac{5}{2}n-5$.

    Therefore the proof is completed.
\end{proof}

\section{Construction of Extremal Graphs}

In the previous sections, we have shown the upper bound of edges for $S_{3,3}$-free planar graphs. Now we shall complete it by demonstrating that this bound is tight.

If $3\leq n\leq 7$, any $n$-vertex maximal planar graph is the extremal graph. If $n=8, 9$, the $4$-regular planar graph is the extremal graph, as shown in Figure~\ref{fig_e_1}$(a)(b)$. Moreover, the extremal graph can also be constructed by some star-block, as shown in Figure~\ref{fig_e_1}$(c)$.

\begin{figure}[ht]
  \centering  \includegraphics[width=0.95\textwidth]{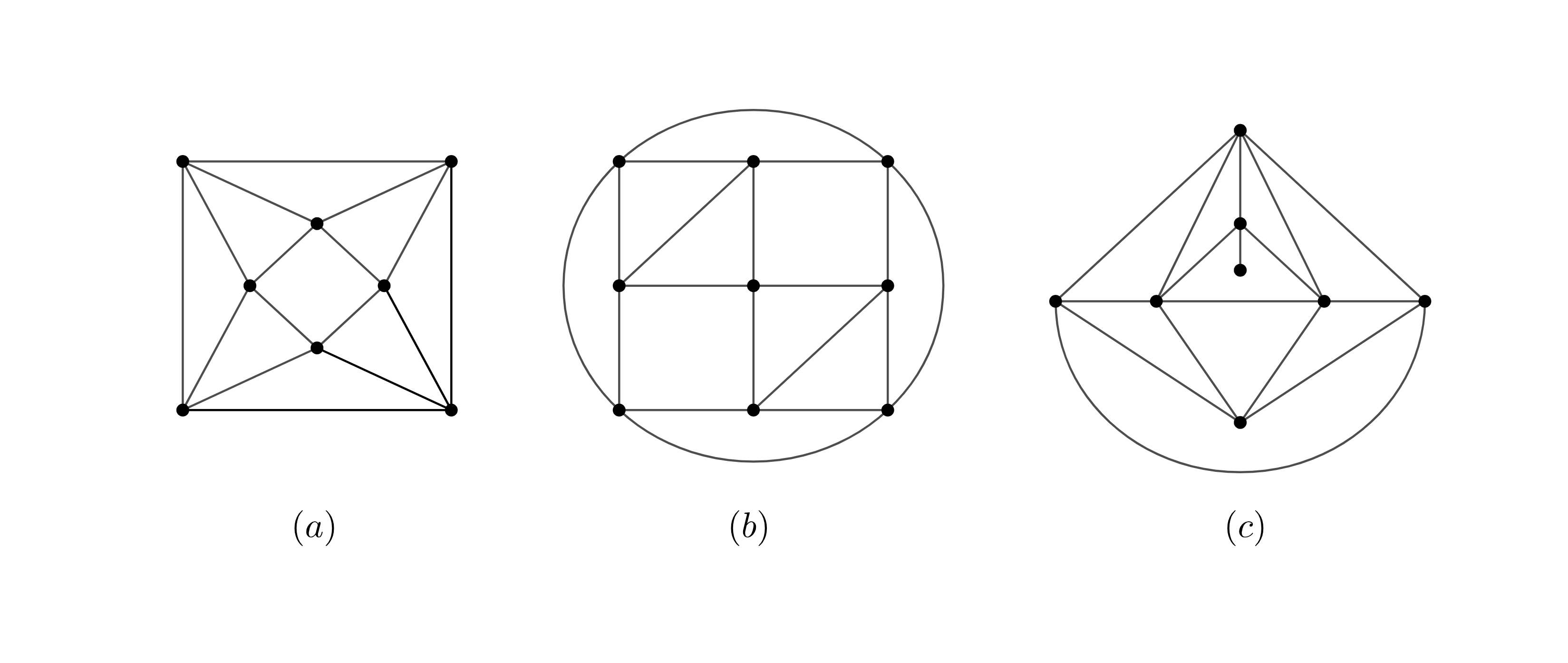}
  \caption{Extremal graphs on $8$ or $9$ vertices. }
  \label{fig_e_1}
\end{figure}

Next we assume that $n\geq 10$. Let $G=G_{1}+G_{2}$ be a star-block partition, where $G_{1}$ is the union of star-blocks and all vertices in $G_{2}$ have degree at most $4$. We also use $t_{0}, t_{1}, t_{2}$ to denote the size of $\mathcal{B}_{0}, \mathcal{B}_{1}, \mathcal{B}_{2}$. Let $t=t_{1}+t_{2}$.

It is known that if $t>0$, the number of edges in extremal graphs must attain the equality in Lemma~\ref{lemma_eq2}.

If $t_{2}>0$, the equation holds when $t_{0}=t_{1}=|V(G_{2})|=0$ and $t_{2}=3$. This means all shared vertices are of degree $3$ and there are exactly $3$ star-blocks. Since each shared vertex is connected with all star-blocks, the number of shared vertices is at most $2$, according to the property of planar graph. It can be checked that the possible star-block is $6$-$5$ edge, or $6$-$4$ edge, or $5$-$5$ edge. However we obtain that $w(B)$ will be reduced strictly here. Thus there does not exist such extremal graph.

Assume that $t_{2}=0$. If $t_{1}>0$, the equation holds when $t_{0}=|V(G_{2})|=0$ and $t_{2}=2$. This means that there are exactly $2$ star-blocks and all shared vertices are of degree at most $3$. Here we construct two extremal graphs by combining different star-blocks. The first extremal graph is obtained by merging a $6$-$6$ edge and a $6$-$5$ edge, as shown in Figure~\ref{fig_e_2}$(a)$. There are $13$ vertices and $27$ edges. 

\begin{figure}[ht]
  \centering  \includegraphics[width=0.95\textwidth]{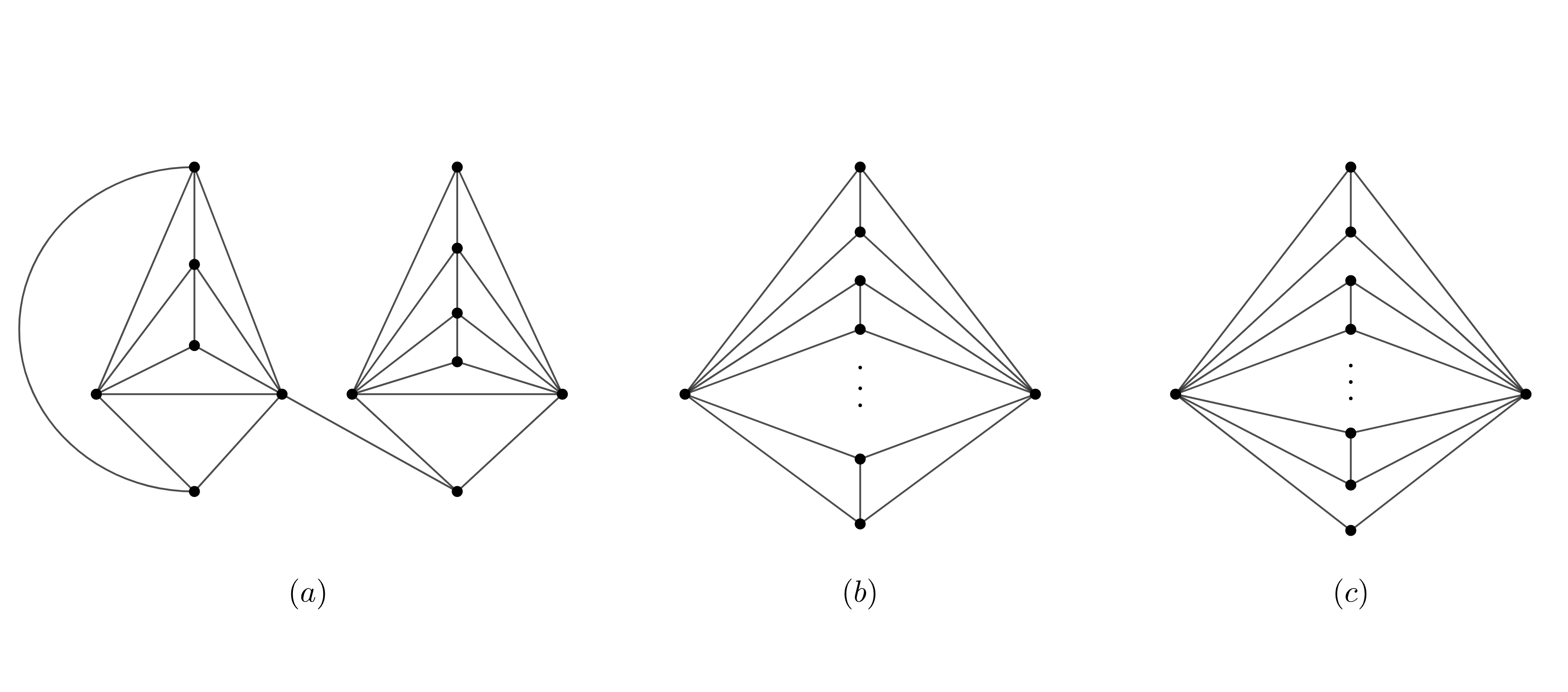}
  \caption{Extremal graphs constructed by merging two star-blocks. }
  \label{fig_e_2}
\end{figure}

The second extremal graph is constructed by two $k$-$3^{-}$ stars, where $k=n-2$. When $n$ is even, the graph obtained from two $k$-$3$ stars is the extremal graph, as shown in Figure~\ref{fig_e_2}$(b)$. When $n$ is odd, the extremal graph is constructed by two $k$-$3^{-}$ stars, where in each star-block, $k-1$ peripheral vertices have degree $3$ and one has degree $2$, as shown in Figure~\ref{fig_e_2}$(c)$. Moreover, $e(G)=\lfloor 5n/2\rfloor-5$.

\begin{figure}[ht]
  \centering  \includegraphics[width=0.65\textwidth]{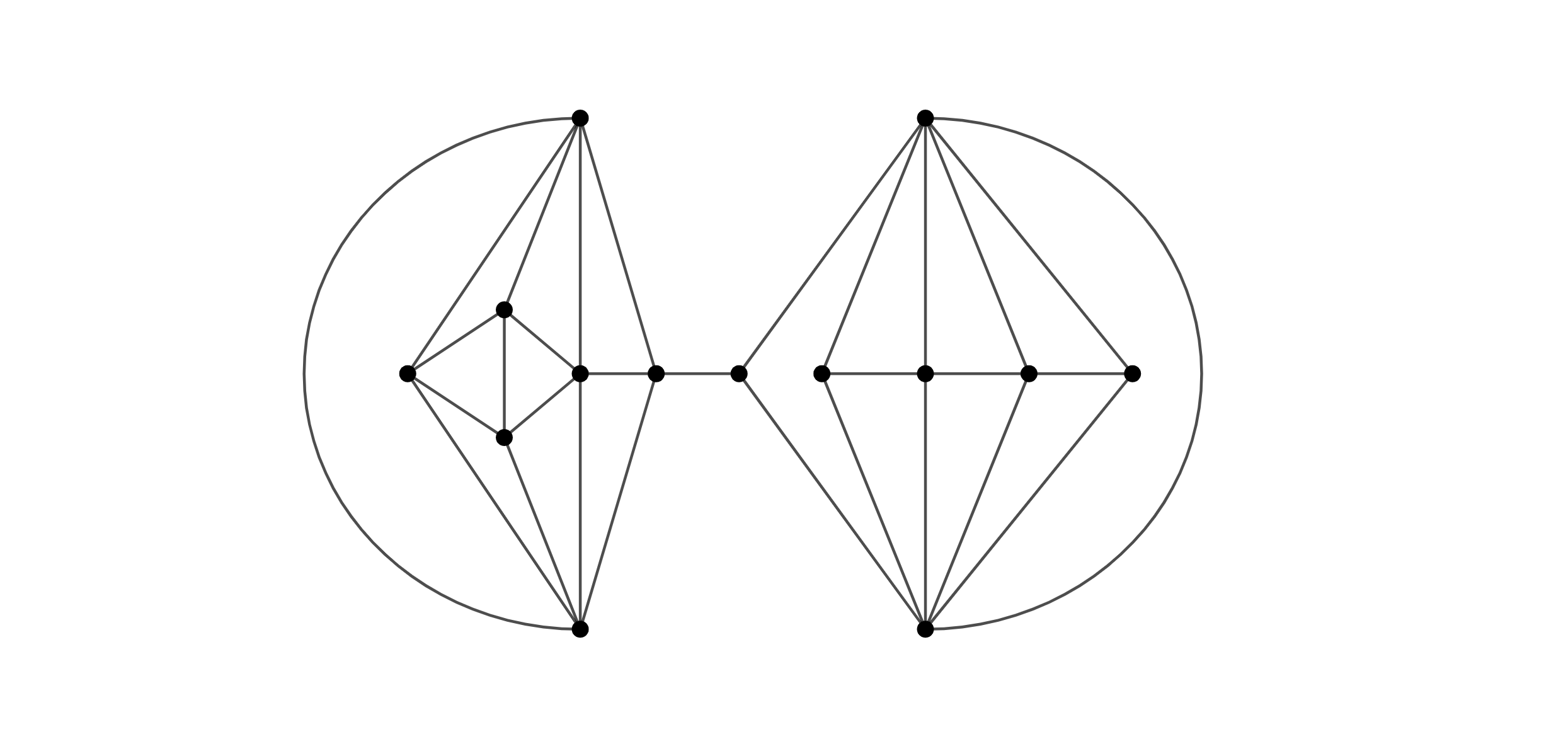}
  \caption{Extremal graphs on $14$ vertices and $30$ edges. }
  \label{fig_e_3}
\end{figure}

Finally, we assume that $ t_{1}=t_{2}=0$. If $t_{0}=1$, the graph in Figure~\ref{fig_65}$(b)$ is an example on $7$ vertices. It remains to discuss the case when $t_{0}=2$. Here we give an extremal graph by connecting two star-blocks, as shown in Figure~\ref{fig_e_3}. Furthermore, if $n=13, 14$, there are extremal graphs that are disconnected. Let $G_{1}, G_{2}$ be the two connected components, where $G_{1}$ is the $7$-vertex planar triangulation and $G_{2}$ is the $6$-vertex or $7$-vertex planar triangulation. It is noticed that there are five $7$-vertex maximal planar graphs, each contains some star-block discussed above.

\section{Remark}

If someone can use a computer to enumerate and verifies Conjecture~\ref{conj} holds for $n\leq 27$, the proof would be much simpler by slightly modifying the induction-based approach of Ghosh, Gy\H{o}ri, Paulos and Xiao~\cite{ghosh2022planar}.

Now we can assume that Conjecture~\ref{conj} holds for $n\leq m-1$, where $m\geq 28$, and let $G$ be an $m$-vertex $S_{3,3}$-free graph. By inductive hypothesis we may assume that $\delta(G)\geq 3$ and $G$ contains no $3$-$3$ edge. Using the same proof of Ghosh, Gy\H{o}ri, Paulos and Xiao~\cite{ghosh2022planar}, we can assume that $G$ contains no $6$-$6$ edge, $6$-$5$ edge, $6$-$4$ edge. Since $G$ is $S_{3,3}$-free, $G$ contains no $7^{+}$-$4^{+}$ edge. Now $G$ is shown in Figure~\ref{fig_rmk}.

\begin{figure}[ht]
  \centering  \includegraphics[width=0.8\textwidth]{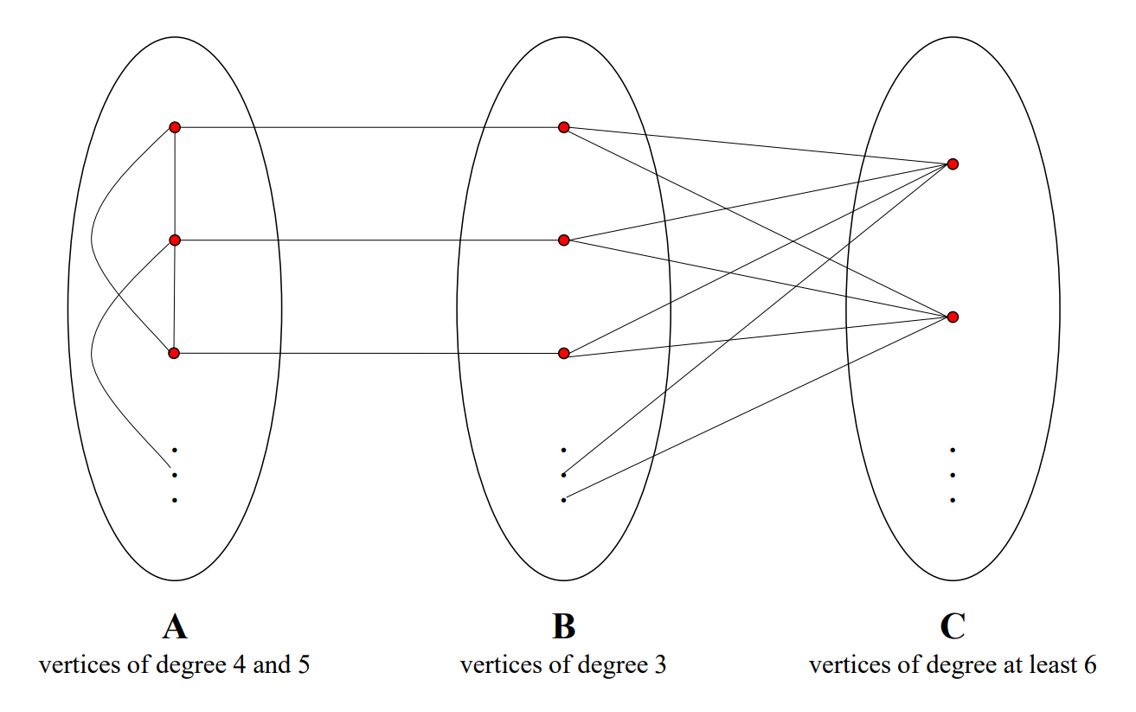}
  \caption{$G$ can be partitioned into three parts $A$, $B$, $C$.}
  \label{fig_rmk}
\end{figure}

Let $m_{3},m_{4},m_{5}$ be the number of vertices of degree $3$, $4$, $5$, respectively. Let $m_{6}$ be the number of vertices of degree at least $6$. Then we have $m_{3}+m_{4}+m_{5}+m_{6}=m$. Let $x$ be the number of edges between $A$ and $B$.

If $m_{6}\geq 2$,
\begin{align*}
    e(G)&=3m_{3}+\frac{1}{2}(4m_{4}+5m_{5}-x)\\
    &=\frac{5}{2}(m-m_{6})+\frac{1}{2}m_{3}-\frac{1}{2}m_{4}-\frac{1}{2}x\\
    &\leq \frac{5}{2}m-\frac{3}{2}m_{6}-\frac{1}{2}m_{4}-2\\
    &\leq \frac{5}{2}m-5,
\end{align*}
where the first inequality follows from counting the number of edges between $B$ and $C$, which is $3m_{3}-x\leq 2(m_{3}+m_{6})-4$, where the right side is the maximum number of edges in the planar bipartite graph induced by $B$ and $C$.

If $m_{6}= 1$,
\begin{align*}
    e(G)&=3m_{3}+\frac{1}{2}(4m_{4}+5m_{5}-x)\\
    &=\frac{5}{2}(m-m_{6})+\frac{1}{2}m_{3}-\frac{1}{2}m_{4}-\frac{1}{2}x\\
    &\leq \frac{5}{2}m-\frac{1}{2}m_{3}-\frac{1}{2}m_{4}-\frac{5}{2},
\end{align*}
where the last inequality follows from counting the number of edges between $B$ and $C$, which is $3m_{3}-x\leq m_{3}$. If $m_{3}+m_{4}\geq 5$ then we are done, otherwise $m_{3}+m_{4}\leq 4$, which means $|B|=m_{3}\leq 4$. However, the vertex in $C$ is adjacent to at least $6$ vertices in $B$, a contradiction.

If $m_{6}= 0$, it is easy to check that every $5$-vertex is adjacent to at least $2$ vertices of degree at most $3$ since $G$ is $S_{3,3}$-free. So we have $2m_{5}\leq 3m_{3}+4m_{4}$. Then $e(G)=\frac{1}{2}(3m_{3}+4m_{4}+5m_{5})=\frac{5}{2}m-\frac{1}{2}m_{4}-m_{3}$. If $2m_{3}+m_{4}\geq 10$ then we are done, so we may assume that $2m_{3}+m_{4}\leq 9$. Now $m=m_{3}+m_{4}+m_{5}\leq \frac{5}{2}m_{3}+3m_{4}\leq 27-\frac{7}{2}m_{3}\leq 27$, a contradiction.

Therefore the proof is completed.

\section*{Acknowledgments} 
Tong Li and Guiying Yan are partially supported by National Natural Science Foundation of China (Grant No. 12301459,  11631014).

\bibliographystyle{abbrv}
\bibliography{main}

\end{sloppypar}
\end{document}